\documentclass[11pt,a4paper,english]{article}

\usepackage[a4paper,margin=2.5cm,vmargin={1.5cm,2.5cm},includeheadfoot]{geometry}
\usepackage[scaled=0.92]{helvet}
\usepackage[latin 1]{inputenc}

\usepackage{amssymb,amscd, amsmath, babel, amstext, amsthm, latexsym}
\usepackage{color}
\usepackage{breqn, dsfont, graphics, enumerate}
\usepackage[font=scriptsize]{caption}
\usepackage{epstopdf}
\usepackage{graphicx}
\usepackage{mathrsfs}

\numberwithin{equation}{section}
\numberwithin{defn}{section}
\numberwithin{thm}{section}
\numberwithin{lem}{section}

\newcommand{\ds}{\text{d}s}

\newcommand{\IN}{{\mathbb{N}}}

\newcommand{\B}{\ensuremath{\mathcal{B}}}

\newcommand{\id}{\text{id}}

\newcommand{\ind}{\mathbf{1}}

\newcommand{\F}{\ensuremath{\mathcal{F}}}
\newcommand{\G}{\ensuremath{\mathcal{G}}}
\newcommand{\Hh}{\ensuremath{\mathcal{H}}}

\newcommand{\FF}{\mathbb{F}}
\newcommand{\GG}{\mathbb{G}}

\newcommand{\R}{\mathbb{R}}

\newcommand{\Ii}{\ensuremath{\mathcal{I}}}

\newcommand{\A}{\ensuremath{\mathcal{A}}}

\newcommand{\DD}{\ensuremath{\mathbb{D}}}

\newtheorem{assumptions}{Assumptions}
\newtheorem{definition}{Definition}
\numberwithin{definition}{section}
\newtheorem{theorem}{Theorem}

\newtheorem{lemma}[theorem]{Lemma}

\newtheorem{corollary}[theorem]{Corollary}

\newtheorem{remark}[theorem]{Remark}

\numberwithin{theorem}{section}

\newcommand{\Ff}{\mathcal{F}}
\newcommand{\Gg}{\mathcal{G}}

\newcommand{\Ss}{\mathcal{S}}

\newcommand{\Ll}{\mathcal{L}}

\newcommand{\Yy}{\mathcal{Y}}
\newcommand{\PF}{\mathcal{P}_{\mathbb{F}}}
\newcommand{\PG}{\mathcal{P}_{\mathbb{G}}}

\title{A Maximum Principle for Mean-Field SDEs with time change}

\author{Giulia Di Nunno\thanks{Department of Mathematics,
University of Oslo, P.O. Box 1053 Blindern, N-0316 Oslo, and Department of Business and Management Science, NHH, Helleveien 30, N-5045 Bergen.
Email: giulian@math.uio.no}  \and
Hannes Haferkorn\thanks{Department of Mathematics,
University of Oslo, P.O. Box 1053 Blindern, N-0316 Oslo. Email: hanneshh@math.uio.no}
}
\date{July 10, 2016}

\begin{document}
\maketitle

\begin{abstract}
Time change is a powerful technique for generating noises and providing flexible models.
In the framework of time changed Brownian and Poisson random measures we study the existence and uniqueness of a solution to a general mean-field stochastic differential equation.  
We consider a mean-field stochastic control problem for mean-field controlled dynamics and we present a necessary and a sufficient maximum principle. For this we study existence and uniqueness of solutions to mean-field backward stochastic differential equations in the context of time change. 
An example of a centralised control in an economy with specialised sectors is provided.

\vspace{2mm}

{\it MS classification:} 60G60, 60H10, 93E20, 91G80
\end{abstract}

\section{Introduction}
The modelling of the interactions and the equilibrium of a large number of agents is an issue in several  fields, e.g. in statistical mechanics with the kinetic theory for gases, in quantum mechanics or chemistry. Equilibria of a large number of agents also naturally appear in biology, in neural networks, and in some economic issues as e.g. systemic risk, commodity markets, and energy related issues. The agents, whatever representing, are assumed symmetric, having similarly shaped dynamics, interacting with the whole population without privileged connections.

The mean-field approach consists of approximating the large number or agents $N$ with a continuum of them $N \longrightarrow \infty$. As clearly presented in e.g. \cite{CDL}, there are two ways to consider such approximation corresponding to different forms of equilibrium.
If the single agents are deciding upon their own individual optimal strategies, then the framework corresponds to a Nash type asymptotic equilibrium. This leads to mean-field games, see e.g. \cite{LL}, \cite{HMC}.
On the other hand another situation is when the decision on the optimal strategy is taken in "centralised form" on the asymptotic common behaviour, which corresponds to a controlled mean-field stochastic differential equation (SDE) and the optimisation problem refers to this dynamics. In this case we have a control problem of a mean-field SDE. See e.g. \cite{AD}, \cite{CD2}. 
The two approaches sketched above are not conceptually equivalent though under some specific conditions the solutions may coincide, see the analysis and examples in \cite{CDL}.
For an overview see e.g. \cite{BFY} and references therein.

\vspace{2mm}
This paper deals with the stochastic control of a mean-field SDE. Our contribution consists in the study of dynamics that are driven by a martingale random field and hence a more general framework than the one considered so far in the literature.
To give a uniform presentation we focus on martingale random fields generated by time changed Brownian and Poisson random fields. However we stress that the first part of the paper, dealing with the existence of solutions of a mean-field SDE, is valid for a general martingale random field with conditionally independent values as defined in \cite{DiE}, see also \cite{walsh}.
The reason for choosing these time changed driving noises comes from the balance between the relative easiness in generating noises in this way and the flexibility of this class of models from the point of view of applications. Classical examples taken from the mathematical finance literature range from the modelling of stochastic volatility to the modelling of abrupt movements in default and more generally in credit risk. In general time changed noises provide the flexibility to cover naturally the modelling of many stochastic phenomena where inhomogeneous behaviour and erratic jump movements are detected. From a mathematical perspective we relate the time changed noises in the representation as doubly stochastic noises as defined here below. We stress that the time changed applied is not necessarily a subordinator, which means that the framework suggested goes well beyond the L\'evy structures.

The specificity of the use of time changed Brownian and Poisson random measures comes in when considering the actual mean-field control problem. In this case, in fact we deal with mean-field backward stochastic differential equations (BSDEs), the solution of which relies on a stochastic integral representation theorem involving the integral with respect to the driving measure only.
The existence of such representation theorems depends on the noise and the information flow fixed on the probability space. It is well known that we can obtain these results for mixtures of Gaussian and Poisson type measures and in \cite{DiS} it is proved for time changed Brownian and time changed Poisson random measures. See also \cite{DiS2} for a specific study on the structure of the doubly stochastic Poisson random noises.

\vspace{2mm}
To summarise in the framework of time change noises, in the sequel we study the solution of a general mean-field SDE in which the coefficients depend not only on the state of the system, but on the distribution of such state. Here we generalise the work of \cite{JMW}, which deals with the L\'evy case. 
Restricting the dynamics and the performance functional to depend on functionals of the distribution of the system, we study a mean-field stochastic control problem by the maximum principle approach.
The mean-field control problems are typically time inconsistent and the approach by maximum principle is a good response to tackle such control problems. 
For this we solve the adjoint equations, studying the mean-field BSDEs driven by time changed noises. In this we extend the work of \cite{BLP}. The mean-field stochastic control problem considered were first studied by \cite{AD} in the Brownian context.
Another way to study maximum principle can be done by the use of Malliavin calculus exploiting the duality between Malliavin derivative and Skorohod integral. For this an adequate extension of the Malliavin calculus needs to be applied. This goes beyond the scopes of the present paper and it is topic of other research.

As illustration of our results we study a centralised control problem in an economy with specialised sectors.

\section{Framework}
Let $(\Omega, \F, P)$ be a complete probability space and $T>0$. 
Let $\lambda:=(\lambda^B,\lambda^H)\in L^1([0,T]\times\Omega;\R_+^2)$ be a two dimensional stochastic process with nonnegative components which are continuous in probability. 
Let $\nu$ be a $\sigma$-finite measure on $\R_0:=\R \setminus \{0\}$ satisfying $\int_{\R_0}z^2\nu(dz)<\infty$. 
Define the random measure $\Lambda$ on $\B([0,T]\times\R)$ as
\begin{align}\label{DefLambda}
 \Lambda(\Delta):=\int\limits_0^T\ind_\Delta(t,0)\lambda^B_t+\int\limits_{\R_0}\ind_\Delta(t,z)\lambda^H_t\nu(dz)dt,\quad \Delta\in\B([0,T]\times\R)
\end{align}
and let the $\sigma$-algebra $\F^\Lambda$ be generated by the values of $\Lambda$ on $[0,T]$. 

\vspace{2mm}
The driving noise for the dynamics we are studying later on is given by the martingale random field $\mu$ on $\B([0,T]\times\R)$ defined by the mixture
$$
\mu(\Delta) := \mu^G ( \Delta \cap [0,T]\times \{0\}) + \mu^P(\Delta \cap [0,T]\times\R_0)
$$
of a doubly stochastic Gaussian random field $\mu^G$ on $[0,T]\times \{0\} \sim [0,T]$ and a doubly stochastic centred Poisson random measure $\mu^P$ on $[0,T]\times \R_0$, such that $\mu^G$ and $\mu^P$ are conditionally independent given $\F^\Lambda$. 
This yields, 
\begin{align}\begin{split}
              &E[\mu(\Delta)|\F^\Lambda]=0,\quad E[\mu(\Delta)^2|\F^\Lambda]=\Lambda(\Delta) \\
              &E[\mu(\Delta_1)\mu(\Delta_2)|\F^\Lambda]=0\text{  for }\Delta_1\text, \Delta_2\text{ disjoint.}
             \end{split}
\end{align}
See e.g. \cite{DiS} for details.
The doubly stochastic noises are set in relationship with time change by the characterisation \cite[Theorem 3.1]{Serfozo} (see also \cite{Grigelionis}). In view of this result $\mu^G$ has the same distribution of a time changed Brownian motion and, for any $B \in \B(\R_0)$, the process $\mu^P([0,\cdot]\times B)$ has the same distribution as a time changed centred pure jump L\'evy process. The corresponding time change processes are independent of the Brownian motion and of the pure jump L\'evy process respectively and they are related to the process $\lambda$.

\bigskip
For any $t$, let $\F^\mu_t$ be the $\sigma$-algebra generated by the values of $\mu$ on $\B([0,t]\times \R)$. 
Then the filtrations $\FF$ and $\GG$ are defined by
\begin{align}
 \F_t&:=\bigcap_{s>t}\F^\mu_s\\
 \Gg_t&:=\F^\mu_t\vee\F^\Lambda.
\end{align}
Remark that, while $\F_0$ is trivial, $\G_0= \F^\Lambda$.
The filtration $\FF$ is relevant for modelling when applications are in view and the control problems will be studied under this information flow. The filtration $\GG$ is technical, better revealing the noise structure and it will serve for computational purposes.
Notice that $\mu$ is a martingale random field with respect to $\GG$ (and also $\FF$) in the sense of \cite[Definition 2.1]{DiE} and an It\^o type non-anticipating integral $I(\phi):=\int_0^T\int_\R\phi_s(z)\mu(ds,dz)$ is then well-defined. See \cite{DiE} (see also \cite{Applebaum} for the specific case of martingale random fields with independently scattered values). 
The space of integrands denoted by $\mathcal{I}:= L^2([0,T]\times\R\times\Omega,\B([0,T]\times\R)\otimes\F,\Lambda\otimes P)$, is the $L^2$-space of the elements admitting a $\GG$-predictable version. 
The norm $\lVert\cdot\rVert_{\mathcal{I}}$ given by
\begin{align*}
 \lVert\phi\rVert_{\mathcal{I}}^2 := E\Big[\int\limits_0^T|\phi_t(0)|^2\lambda^B_t+\int\limits_{\R_0}\phi_t(z)\lambda^H_t\nu(dz)dt\Big].
\end{align*}
We recall that $\GG$-predictable refers to the predictable $\sigma$-algebra
$$
 \PG :=\sigma((s,u]\times B\times A:\,0\leq s<u\leq T,\, A\in\Gg_s, B\in\B(\R)),
$$
For later use we introduce also $\PF \subseteq \PG$ as
$$
 \PF :=\sigma((s,u]\times B\times A:\,0\leq s<u\leq T,\, A\in\F_s, B\in\B(\R)).
$$

When considering the stochastic integration with respect to $\mu$ and $\GG$, we have a stochastic integral representation theorem of the following form: {\it for any $\G_T$-measurable $F\in L^2(\Omega,\F,P)$, there exists $\phi \in \mathcal{I}$ such that
$$
F = F^0 \oplus \int_0^T\int_\R \phi_t(z) \mu(dt,dz) \quad \textrm{for}\quad F^0=E[F\vert \F^\Lambda], 
$$
where the integrand $\phi$ can be explicitly expressed in terms of the non-anticipating derivative.} 
See \cite[Definition 3.4, Theorem 3.1]{DiE} (see also \cite[Theorem 3.3]{DiS}).


\section{Mean-field SDEs}
Following a classical approach by the fixed point theorem, yet adapted to the present framework, we prove the existence of a strong solution to the mean-field SDE
\begin{align}\label{mfSDE}
X_t=x+\int\limits_0^tb(s,X_{s-},\Ll_{X_{s}})ds+\int\limits_0^t\int\limits_\R\kappa(s,z,X_{s-},\Ll_{X_{s}})\mu(ds,dz),
\end{align}
for appropriate $b:[0,T]\times\R\times M_0(\R)\times\Omega\rightarrow\R$ and $\kappa:[0,T]\times\R\times\R\times M_0(\R)\times\Omega\rightarrow\R$, where $M_0(S)$ denotes the space of probability measures on the topological space $S$ equipped with the Borel $\sigma$-algebra and, for all $s$, $\Ll_{X_{s}}$ denotes the law of $X_{s}$. 
Mean-field SDEs driven by Brownian or L\'evy noises were studied in e.g. \cite{AD} and \cite{JMW}.
Note that the results of this section are valid for any martingale random field with square integrable conditionally independent values as in \cite[Definition 2.1]{DiE}. 
To keep the exposition uniform throughout the paper we present the results for the time changed noises. In this case, for the filtration $\GG$, we have that, for all $\B \in \B(\R)$,
$\langle \mu([0,\cdot]\times B)\rangle_t = \Lambda([0,t]\times B)$, $\in [0,T]$.
See \cite[Theorem 2.1]{DiE}.

\bigskip
Hereafter we consider two metric spaces with Wasserstein metric. The first is the space $M_2(\R)$ of elements $Q\in M_0(\R)$ such that 
$
 \int_\R|r|^2 Q(dr)<\infty
$
equipped with the metric $d_\R$ given by the infimum
$$
d_{\R}(P,Q) =\inf_R\Big(\int\limits_{\R^2}|v-w|^2R(dv,dw)\Big)^{\frac{1}{2}}
$$
over all measures $R\in M_0(\R^2)$ with marginals $P$ and $Q$, that is $R(U\times S)=P(U)$ and $R(S\times U)=Q(U)$, for all $U\in \B(\R)$.

Let $\DD$ denote the space of all real c\`adl\`ag functions on $[0,T]$ equipped with the $\sup$-norm $\lVert\cdot\rVert_\infty$. 
As above we define the metric space
$M_2(\DD)$ of elements $Q\in M_0(\DD)$ such that $\int_\DD\lVert Y\rVert^2_\infty Q(dY)<\infty$, equipped with the metric
$$
d_{\DD}(P,Q)=\inf_R \Big(\int\limits_{\DD^2}\lVert V-W\rVert_\infty^2R(dV,dW)\Big)^{\frac{1}{2}} 
$$
where the infimum is taken over all $R\in M_0(\DD^2)$ with marginals $P$ and $Q$.

Let $Q \in M_2(\DD)$ and, for every $s$, let $Q_s$ be the probability measure corresponding to:
$$
Q_s(A) = Q\{ Y\in \DD :\, Y(s) \in A\}
$$
At first we study an SDE of type:
\begin{align}\label{mfSDE-Q}
 X_t=x+\int\limits_0^t b(s,X_{s-},Q_{s})ds+\int\limits_0^t\int\limits_\R\kappa(s,z,X_{s-},Q_{s})\mu(ds,dz),
\end{align}
and then we specialise the result to \eqref{mfSDE}.
To guarantee that the terms in the above equation are well-defined, we summarise some results.
\begin{lemma}
\label{cagladicityQminus}
For all $s\in [0,T]$, the probability measure $Q_s \in M_2(\R)$ and the function $s \longmapsto Q_{s}$ is c\`adl\`ag and Borel measurable.
\end{lemma}

\begin{proof}
The proof is based on direct arguments, which can also be partially retrieved within the proof of \cite[Proposition 1.2]{JMW}. Hereafter follows a sketch. The proof of $Q_s\in M_2(\R)$ exploits the domination by the sup-norm. The c\`adl\`ag property is obtained by dominated convergence. For this we observe that 
$Q_{s-}$ is the weak limit of $Q_u$ for $u \uparrow s$ and it is also
$$
Q_{s-} (A) = Q\{ Y\in \DD :\, Y(s-) \in A\}.
$$
The measurability is proved by point-wise approximation taking, e.g., the sequence of step functions $F_n:[0,T]\rightarrow M_2(\R)$ of type
\begin{align*}
F_n(t):=\sum_{j=1}^n Q_{\frac{j}{n}T}  \ind_{[\frac{j-1}{n}T,\frac{j}{n}T)}(t) 
\end{align*}
Here we make use of the c\`adl\`ag property proved earlier.
\end{proof}

For later use, we introduce the notation
$
\Ss^\FF_2
$ 
for the $\FF$-adapted stochastic processes $Y$ such that $\Vert Y \Vert^2_{S_2}=E\Big[\sup_{t \in [0,T]} \vert Y_t \vert^2 \Big] < \infty$.
Furthermore, for any $s\in[0,T]$, we introduce the notation $\Vert \cdot\Vert_{\lambda_s}$ for the seminorm defined ($\omega$-wise) by
\begin{align*}
 \lVert\alpha\rVert_{\lambda_s}^2 := |\alpha(0)|^2\lambda^B_s+\int_{\R_0}|\alpha(z)|^2\lambda^H_s \nu(dz).
\end{align*}

\begin{assumptions}\hspace{10cm}
 \begin{itemize}
  \item[$(E1)$] The real functions $b(s,x,\Yy, \omega)$ and $\kappa(s,z,x,\Yy, \omega)$, $s\in [0,T]$
  $x\in \R$, $z \in \R$, $\Yy\in M_2(\R)$, $\omega\in \Omega$
  are $\PF\otimes\B(\R)\otimes\B(M_2(\R))$-measurable.
  \item[$(E2)$] The functions $(x,\Yy)\longmapsto b(s,x,\Yy,\omega)$ and $(x,\Yy)\longmapsto \kappa(s,\cdot,x,\Yy,\omega)$ are globally Lipschitz, i.e. for all $s,\omega$ there exists a constant $C\geq 0$ such that
  \begin{align*}
   &|b(s,x_1,\Yy_1,\omega)-b(s,x_2,\Yy_2,\omega)|+\lVert\kappa(s,\cdot,x_1,\Yy_1,\omega)-\kappa(s,\cdot,x_2,\Yy_2,\omega)\rVert_{\lambda_s} \\
   &\leq C\big(|x_1-x_2|+d_{\R}(\Yy_1,\Yy_2)\big)\quad \textrm{for all } x_1, x_2 \in \R, \Yy_1, \Yy_2 \in M_2(\R)
  \end{align*}
  \item[$(E3)$] For the Dirac measure at 0, we have
  \begin{align*}
   E\Big[\int_0^T|b(s,0,\delta_0)|^2+\lVert\kappa(s,\cdot,0,\delta_0)\rVert_{\lambda_s}^2ds\Big]<\infty.
  \end{align*}
 \end{itemize}
\end{assumptions}


\begin{remark}\label{kappainI}
Under assumptions $(E1)$ and $(E2)$ we have that, for any $\FF$-predictable process $x_t$, $t\in[0,T]$ and $Q_{t}$ as defined above, the stochastic process
 \begin{align}
  (t, z, \omega) \longmapsto\kappa(t,z,x_t(\omega),Q_{t}, \omega)
 \end{align}
 is predictable, i.e. $\PF$-measurable.
To see this it is enough to observe that $(t,z, \omega)\longmapsto x_t(\omega)$ is $\PF$-measurable and then proceed by composition of measurable functions.
\end{remark}


\begin{theorem}\label{ExistenceAndUniquenessSDE}
 Assume $(E1)-(E3)$.
 For any fixed probability measure $Q\in M_2(\DD)$, the SDE \eqref{mfSDE-Q}:
 \begin{align}\label{SDEfixedMeasure}
  X^Q_t=x+\int\limits_0^tb(s,X^Q_{s-},Q_s)ds+\int\limits_0^t\int\limits_\R\kappa(s,z,X^Q_{s-},Q_s)\mu(ds,dz),
 \end{align}
 has a unique c\`adl\`ag solution in $\Ss^\FF_2$. 
\end{theorem}

\begin{proof}
 The proof is organised in two steps. First, we show that, if there is a c\`adl\`ag solution $X^Q$ to \eqref{SDEfixedMeasure}, then it necessarily lies in the Banach space $\Ss^\FF_2$. 
 In a second step, we use Banach's fixed point theorem in order to obtain existence and uniqueness. To do so, we define the mapping $F:\Ss^\FF_2\rightarrow \Ss^\FF_2$, by
 \begin{align*}
  F(X)_t:=x+\int\limits_0^tb(s,X_{s-},Q_s)ds+\int\limits_0^t\int\limits_\R\kappa(s,z,X_{s-},Q_s)\mu(ds,dz),
 \end{align*}
 and show that it is a contraction.

 \vspace{2mm}
 \textit{Step 1:} We prove that any c\`adl\`ag solution $X^Q$ to \eqref{SDEfixedMeasure} necessarily lies in $S^2_\FF$. For this, we consider the increasing sequence of stopping times $\tau_n:=\inf\{t\in[0,T]:\,X^Q_t>n\}$, $n\in\IN$. Since $X^Q$ is c\`adl\`ag, we have $X^Q_{s-}\leq n$ for each $s\leq\tau_n$. Observe
 \begin{align*}
  \lVert X^Q_{\cdot\wedge\tau_n}\rVert_{S^2}^2
  =E \Big[\sup_{t\in[0,T]}|X^Q_{t\wedge\tau_n}|^2\Big]
  \leq 3|x|^2+3TE\Big[\int\limits_0^{T\wedge\tau_n}|b(s,X^Q_{s-},Q_s)|^2ds\Big]+3E\Big[\sup_{t\in[0,T]}|M_{t\wedge\tau_n}|^2\Big],
 \end{align*}
 where $M_t:=\int\limits_0^t\int\limits_\R\kappa(s,z,X^Q_{s-},Q_s)\mu(ds,dz)$. By the Burkholder-Davis-Gundy inequality we have
 \begin{align*}
  E\Big[\sup_{t\in[0,T]}|M_{t\wedge\tau_n}|^2\Big]
  =C_1E\Big[[M]_{T\wedge\tau_n}\Big]=C_1E\Big[\int\limits_0^{T\wedge\tau_n}\lVert \kappa(s,\cdot,X^Q_{s-},Q_s)\rVert^2_{\lambda_s}ds\Big].
 \end{align*}
 Therefore, exploiting (E2) and (E3), we get
 \begin{align}
  E\Big[\sup_{t\in[0,T]} & |X^Q_{t\wedge\tau_n}|^2\Big]\nonumber
  \leq 3|x|^2+3(T\vee C_1)E\Big[\int\limits_0^{T\wedge\tau_n}|b(s,X^Q_{s-},Q_s)|^2+\lVert \kappa(s,\cdot,X^Q_{s-},Q_s)\rVert^2_{\lambda_s}ds\Big]\nonumber\\
  &\leq 3|x|^2+3(T\vee C_1)E\Big[\int\limits_0^T2|b(s,0,\delta_0)|^2+2\lVert \kappa(s,\cdot,0,\delta_0)\rVert^2_{\lambda_s}ds\Big]\nonumber\\
  &+3(T\vee C_1)E\Big[\int\limits_0^{T\wedge\tau_n}2|b(s,X^Q_{s-},Q_s)-b(s,0,\delta_0)|^2+2\lVert \kappa(s,\cdot,X^Q_{s-},Q_s)-\kappa(s,\cdot,0,\delta_0)\rVert^2_{\lambda_s}ds\Big]\nonumber\\
  &\leq 3|x|^2+6(T\vee C_1)E\Big[\int\limits_0^T|b(s,0,\delta_0)|^2+\lVert \kappa(s,\cdot,0,\delta_0)\rVert^2_{\lambda_s}ds\Big]\nonumber\\
  &+6(T\vee C_1)C^2E\Big[\int\limits_0^{T\wedge\tau_n}|X^Q_{s-}|^2+d_\R(Q_s,\delta_0)^2ds\Big].\label{Gronwall1}
 \end{align}
 Moreover, observe that
 \begin{align*}
  d_\R(Q_s,\delta_0)^2
  &\leq\int\limits_{\R^2}|v-w|^2Q_s(dv)\delta_0(dw)
  \leq\int\limits_{\DD}\lVert Y\rVert_{\infty}^2Q(dY)<\infty.
 \end{align*}
Substituting this in \eqref{Gronwall1} and exploiting $|X^Q_{s-}|^2\leq n^2$, for all $s\leq T\wedge\tau_n$, we get
 \begin{align*}
  E\Big[\sup_{t\in[0,T]}|X^Q_{t\wedge\tau_n}|^2\Big]
  &\leq 3|x|^2+6(T\vee C_1)E\Big[\int\limits_0^T|b(s,0,\delta_0)|^2+\lVert \kappa(s,\cdot,0,\delta_0)\rVert^2_{\lambda_s}ds\Big]\\
  &\quad+6(T\vee C_1)C^2T\Big(n^2+\int\limits_{\DD}\lVert Y\rVert_{\infty}^2Q(dY)\Big)
  <\infty.
 \end{align*}
Hence the function $s\longmapsto E[\sup_{t\in[0,s]}|X^Q_{t\wedge\tau_n}|^2]$ is Lebesque integrable. In fact \begin{align*}
  \int\limits_0^T \Big|E\Big[\sup_{t\in[0,s]}|X^Q_{t\wedge\tau_n}|^2\Big]\Big| ds
  \leq T E \Big[\sup_{t\in[0,T]}|X^Q_{t\wedge\tau_n}|^2\Big]<\infty.
 \end{align*}
 The integrability allows us now to apply Gronwall's inequality to \eqref{Gronwall1} since
 \begin{align*}
  E\Big[\sup_{t\in[0,T]}|X^Q_{t\wedge\tau_n}|^2\Big]
\leq   K_1 + K_2 \int\limits_0^{T}E\Big[\sup_{t\in[0,s]}|X^Q_{t\wedge\tau_n}|^2\Big] ds
 \end{align*}
 with the finite positive constants
 \begin{align*}
  K_1&:=3|x|^2+6(T\vee C_1)E\Big[\int\limits_0^T|b(s,0,\delta_0)|^2+\lVert \kappa(s,\cdot,0,\delta_0)\rVert^2_{\lambda_s}ds\Big]+6(T\vee C_1)C^2T\int\limits_{\DD}\lVert Y\rVert_{\infty}^2Q(dY)\\
  K_2&:=6(T\vee C_1)C^2.
 \end{align*}
 Thus we obtain
$
  E[\sup_{t\in[0,T]}|X^Q_{t\wedge\tau_n}|^2]\leq K_1e^{K_2T}<\infty.
$
By monotone convergence we can conclude $\lVert X^Q\rVert_{S^2}^2 < \infty$.

 \vspace{2mm}
 \textit{Step 2:} 
 Here we see that for any $X\in \Ss^\FF_2$ the value $F(X)$ is well-defined. 
Since  $t\longmapsto$, $(X_{s-})_{s\in[0,T]}$ is c\`agl\`ad (and therefore predictable as it is adapted), thanks to Remark \ref{kappainI} we can guarantee that $\phi_s(\cdot):=\kappa(s,\cdot,X_{s-},Q_s)$ is predictable. 
 
For any $X\in \Ss^\FF_2$ and being $\kappa$ Lipschitz, we get that $\lVert\phi\rVert_{\mathcal{I}}<\infty$. This implies that $F$ is well-defined on the entire  $\Ss^\FF_2$ and the stochastic process $\int_0^t\int_\R\phi_s(z)\mu(ds,dz)$, $t\in [0,T]$, is a martingale (see \cite{DiE}, Remark 3.2)). 
Since $\FF$ is right-continuous, then the martingale process of the integrals has a c\`adl\`ag version (see, e.g. Theorem 6.27 (ii) in \cite{K}).
Then, w.l.o.g., we choose $F(X)$ to be c\`adl\`ag (the integral w.r.t. $ds$ is continuous). 
By the same arguments as in Step 1, with the only difference being that we exploit $E[\sup_{t\in[0,T]}|X_t|^2]<\infty$ instead of using the Gr\"onwall inequality, we can see that $E[\sup_{t\in[0,T]}|F(X)_t|^2]<\infty$. This proves that $F$ indeed maps into $\Ss^\FF_2$.

 \vspace{1mm}
Let $F^{\circ 0}=\id$, i.e. $F^{\circ 0}(X)=X$, and let $F^{\circ n}$
 denote the $n^{\text{th}}$ composition of $F$. 
Now we show that, for $n$ large enough, this is a contraction on $\Ss^\FF_2$. 
By the same reasoning as above, we have
 \begin{align*}
  \lVert F^{\circ n}(X)-F^{\circ n}(Y)\rVert_{S_2}^2
  &=E\Big[\sup_{t\in[0,T]}|F(F^{\circ n-1}(X))_t-F(F^{\circ n-1}(Y))_t|^2\Big]\\
&=E\Big[\sup_{t\in[0,T]}\Big|\int\limits_0^tb(s,F^{\circ n-1}(X)_{s-},Q_s)-b(s,F^{\circ n-1}(Y)_{s-},Q_s)ds
 \\
  &\quad+\int\limits_0^t\int\limits_\R\kappa(s,z,F^{\circ n-1}(X)_{s-},Q_s)-\kappa(s,z,F^{\circ n-1}(Y)_{s-},Q_s)\mu(ds,dz)\Big|^2\Big]
   \end{align*}
  \begin{align*}
   \hspace{2cm} &\leq 2(T\vee C_1)E\Big[\int\limits_0^T|b(s,F^{\circ n-1}(X)_{s-},Q_s)-b(s,F^{\circ n-1}(Y)_{s-},Q_s)|^2ds\\
  &\quad+\int\limits_0^T\lVert\kappa(s,z,F^{\circ n-1}(X)_{s-},Q_s)-\kappa(s,z,F^{\circ n-1}(Y)_{s-},Q_s)\rVert_{\lambda_s}^2ds\Big]\\
  &\leq 2(T\vee C_1)C^2\int\limits_0^TE\Big[\sup_{t\leq s}|F^{\circ n-1}(X)_{t}-F^{\circ n-1}(Y)_{t}|^2\Big]ds.
 \end{align*}
 By iteration down to $0$, making use of $F^{\circ 0}=\id$ and Fubini's theorem, we get
 \begin{align*}
  \lVert F^{\circ n}(X)-F^{\circ n}(Y)\rVert_{S_2}^2&\leq 2^n(T\vee C_1)^nC^{2n}\int\limits_0^T\int\limits_0^{t_n}\cdots\int\limits_0^{t_2}E\Big[\sup_{t\leq t_1}|X_{t}-Y_{t}|^2\Big]dt_1\cdots dt_{n-1}dt_n\\
  &=2^n(T\vee C_1)^nC^{2n}\int\limits_0^T\int\limits_{t_1}^T\cdots\int\limits_{t_{n-1}}^TE\Big[\sup_{t\leq t_1}|X_{t}-Y_{t}|^2\Big]dt_n\cdots dt_2dt_1\\
  &=2^n(T\vee C_1)^nC^{2n}\int\limits_0^TE\Big[\sup_{t\leq t_1}|X_{t}-Y_{t}|^2\Big]\int\limits_{t_1}^T\cdots\int\limits_{t_{n-1}}^Tdt_n\cdots dt_2dt_1\\
  &=2^n(T\vee C_1)^nC^{2n}\int\limits_0^TE\Big[\sup_{t\leq t_1}|X_{t}-Y_{t}|^2\Big]\frac{(T-t_1)^{n-1}}{(n-1)!}dt_1\\
  &\leq\frac{2^n(T\vee C_1)^nC^{2n}T^n}{n!}E\Big[\sup_{t\leq T}|X_{t}-Y_{t}|^2\Big].
 \end{align*}
Since
 \begin{align*}
  \sum_{n=0}^\infty\frac{2^n(T\vee C_1)^nC^{2n}T^n}{n!}=\exp(2(T\vee C_1)C^{2}T)<\infty,
 \end{align*}
 the term $\frac{2^n(T\vee C_1)^nC^{2n}T^n}{n!}$ vanishes as $n$ goes to infinity. 
 Thus, for $n$ large enough, we have
 \begin{align*}
  \lVert F^{\circ n}(X)-F^{\circ n}(Y)\rVert_{S_2}^2\leq\frac{1}{2}\lVert X-Y\rVert_{S_2}^2
 \end{align*}
and $F^{\circ n}$ is a contraction.
By Banach's fixed point theorem there exists one unique point $X^Q\in \Ss^\FF_2$ such that $X^Q=F^{\circ n}(X^Q)$. 
This is then also a fixed point for $F$. Observe that $F(X^Q)=F(F^{\circ n}(X^Q))=F^{\circ n}(F(X^Q))$. Hence $F(X^Q)$ is fixed point for $F^{\circ n}$. By uniqueness of the fixed point we have then $F(X^Q)=X^Q$. 
By this we conclude.
\end{proof}

We turn now to the study of \eqref{mfSDE}.

\begin{theorem}\label{ExistenceAndUniquenessmfSDE}
Assume $(E1)-(E3)$. The mean-field SDE \eqref{mfSDE} has exactly one non-exploding c\`adl\`ag solution $X$ in the sense that $X\in \Ss^\FF_2$, i.e.
  $E\Big[\sup_{t\in[0,T]}|X_t|^2\Big]<\infty.
 $
\end{theorem}

We remark that in the case of the SDE \eqref{SDEfixedMeasure}, being $Q \in M_2(\DD)$ fixed, we could deduce that the unique solution was necessarily an element of $\Ss^\FF_2$. 
For the SDE \eqref{mfSDE} this is not the case. 
Hence we restrict the study to the non-exploding solutions.

\begin{proof}
Relying on Theorem \ref{ExistenceAndUniquenessSDE} the arguments follow the same steps as 
\cite[Proposition 1.2]{JMW}, which is though formulated for L\'evy processes only. Hereafter, we only sketch the main steps.

\vspace{1mm}
First we observe that, having restricted the study to non-exploding solutions $X$ we have
 \begin{align*}
  \int\limits_\DD\lVert Y\rVert^2_\infty \Ll_X(dY)=\int\limits_\Omega\sup_{t\in[0,T]}|X_t(\omega)|^2P(d\omega)=E\Big[\sup_{t\in[0,T]}|X_t|^2\Big]<\infty.
 \end{align*}
 Therefore necessarily $\Ll_X\in M_2(\DD)$. Define the function $\Phi:M_2(\DD)\rightarrow M_2(\DD)$ such that $Q\longmapsto\Ll_{X^Q}$, where $X^Q$ is the solution of \eqref{SDEfixedMeasure} corresponding to the input measure $Q$. By Theorem \ref{ExistenceAndUniquenessSDE}, $X^Q\in \Ss^\FF_2$, which implies $\Ll_X^Q\in M_2(\DD)$). Observe that $X^Q$ is a non-exploding solution of \eqref{mfSDE} if and only if $Q$ is a fixed point of $\Phi$.
Finally we show that $\Phi$ is a contraction. This is done following the same arguments as for Step 2 in the proof of Theorem \ref{ExistenceAndUniquenessSDE}.
\end{proof}

\section{Mean-field BSDEs}
In the sequel we intend to study the stochastic control problem
$$
\sup_u E\Big[ \int_0^T f(s, \lambda_s, X^u_{s-}, E[\varphi(X^u_s)], u_s) ds + g(X^u_T, E [\chi(X_T^u)]) \Big]
$$
via a maximum principle. Hence we deal with the adjoint equation associated to the Hamiltonian function, which follows backward dynamics. 
Before entering the core of the issue we present the necessary results related to mean-field BSDEs.
We follow the approach of \cite{BLP} and exploit the techniques suggested in \cite{DiS} and \cite{DiS2} for time changed L\'evy noises.

\vspace{2mm}
First we introduce some notation. 
For any random variable $X$ on $(\Omega,\F,P)$, we draw its independent copy, which is denoted by $X'$. More precisely, we consider the product probability space $(\Omega^2, \F^\otimes, P^\otimes)$
= $(\Omega\times\Omega,\F\otimes\F,P\otimes P)$ where we can identify the original random variable $X$ with
$$ 
X(\tilde\omega,\omega):=X(\omega)
$$
and its independent copy $X'$ with 
$$
X'(\tilde\omega,\omega):=X(\tilde\omega).
$$
Moreover, we define the functional $\mathbb{E}: L^1(\Omega^2,\R)\longrightarrow \R$:
$$
 \mathbb{E}[Y]:=\int\limits_{\Omega^2}Y(\tilde\omega,\omega)P^{\otimes2}(d\tilde\omega,d\omega)
 $$
 and the operator $E': L^1(\Omega^2,\R)\longrightarrow L^1(\Omega,\R)$:
 $$
 E'[Y](\omega):=\int\limits_\Omega Y(\tilde\omega,\omega)P(d\tilde\omega).
 $$
In particular, for the random variable $X$ and its copy $X'$ we have that
\begin{equation}
\label{correspondence}
E'[X]=X \textrm { and } E'[X']=E[X].
\end{equation}
Let us also introduce the spaces
$L^2_{ad}(\GG)$ and $L^2_{pred}(\GG)$ of $\GG$-adapted and, correspondingly, $\GG$-predictable stochastic processes such that $E[\int_0^T|Y_s|^2ds]<\infty $.
Also we define $ \Ss^\GG_2$ as the space of of $\GG$-adapted stochastic processes such that $\lVert Y\rVert_{\Ss_2}^2=E[\sup_{s\leq T}|Y_s|^2]<\infty.$
Furthermore we define $L^2_{pred}(\F\otimes\GG)$ of $\F\otimes\GG$-predictable stochastic process such that $\mathbb{E} \Big[\int_0^T|Y_s|^2ds \Big]<\infty$.
Here $\F\otimes\GG$ is the filtration given by $\F\otimes \G_t$, $t \in [0,T]$

Finally, let
$$
L^2(\delta_0+\nu):=\Big\{\alpha:\R\rightarrow\R:\: \lVert\alpha\rVert^2:=|\alpha(0)|^2+\int_{\R_0}|\alpha(z)|^2\nu(dz)<\infty\Big\}.
$$
In this framework we study
existence and uniqueness of the $\GG$-adapted solutions of the BSDE of type:
\begin{align}\label{GeneralmfBSDE}
 \begin{cases}
  dY_t&=E'\Big[h(t,\lambda_t,\lambda'_t,Y_t,Y'_t,Z_t(\cdot),Z'_t(\cdot))\Big]dt+\int_\R Z_t(z)\mu(dt,dz)\\
  Y_T&=F
 \end{cases}
\end{align}
for appropriate conditions on $F$ and
$h:[0,T]\times \R^2\times \R^2 \times \big(L^2(\delta_0+\nu)\big)^2\times \Omega^2 \longrightarrow \R$. 

\vspace{2mm}
For any $(Y,Z) \in L^2_{ad}(\GG) \times\Ii$ define the real function
\begin{align}
 \tilde h(t,l,y,z(\cdot)) := E'\Big[h\Big(t,l,\lambda'_t,y,Y_{t}',z(\cdot),Z_{t}(\cdot) '\Big)\Big],
 \quad t \in [0,T], l,y \in \R, z\in L^2(\delta_0+\nu).
\end{align}

\begin{assumptions}\label{AssumptionsmfBSDE}\hspace{10cm}
 \begin{itemize}
  \item[(C1)] $F\in L^2(\Omega, \F,P)$ is $\G_T$-measurable
  \item[(C2)] 
  $Y_1,Y_2 \in L^2_{ad}(\GG)$ and all $Z_1,\,Z_2\in\Ii$ the stochastic process
  $h\Big(t,\lambda_t,\lambda'_t,Y_{1,t},Y_{2,t}', Z_{1,t}(\cdot),Z_{2,t}(\cdot))' \Big)$, $t \in [0,T]$, is 
  $\F\otimes\GG$-adapted
     \item[(C3)] For all $y_1,\,y_1',\,y_2,\,y_2'\in\R$, $z_1,\,z_1',\,z_2,\,z_2'\in L^2(\delta_0+\nu)$ and $(\tilde \omega, \omega) \in \Omega^2$, there exists a constant $K>0$ such that
  \begin{align*}
  &|h(t,\lambda_t,\lambda'_t,y_1,y_1',z_1,z_1',\tilde\omega,\omega)-h(t,\lambda_t,\lambda'_t,y_2,y_2',z_2,z_2',\tilde\omega,\omega)|\\
  &\leq K \Big(|y_1-y_2|+|y_1'-y_2'|+\lVert z_1-z_2\rVert_{\lambda_t(\omega)}+\lVert z_1'-z_2'\rVert_{\lambda'_t(\tilde\omega,\omega)}\Big)
  \end{align*}
  \item[(C4)] the stochastic porcess $h(t,\lambda_t,\lambda'_t,0,0,0,0)$, $t\in [0,T]$ belongs to $L^2_{pred}(\F\otimes\GG)$
  \item[(C5)] For all $(Y,Z)\in L^2_{ad}(\GG)\times\Ii$ the stochastic process $\tilde h(t,\lambda_t,0,0)$ belongs to $L^2_{pred}( \GG)$.
 \end{itemize}
\end{assumptions}

\begin{theorem}
\label{solutionBSDE}
Assume $(C1)-(C5)$. Then there exists a unique $\GG$-adapted solution $(Y,Z) \in \Ss^\GG_2 \times \Ii$ to the mean-field BSDE \eqref{GeneralmfBSDE}. 
\end{theorem}

\begin{proof}
First we study the following BSDE for any given couple $(Y^{(0)},Z^{(0)})\in L^2_{ad}(\GG)\times\Ii$:
\begin{align}\label{RelaxedmfBSDE}
 \begin{cases}
  dY^{(1)}_t&=E'\Big[h\Big(t,\lambda_t,\lambda'_t,Y^{(1)}_t,(Y^{(0)}_t)',Z^{(1)}_t(\cdot),(Z^{(0)}_t(\cdot))'\Big)\Big]dt+\int_\R Z^{(1)}_t(z)\mu(dt,dz)\\
  &=\tilde h(t,\lambda_t,Y^{(1)}_t,Z^{(1)}_t(\cdot))dt+\int_\R Z^{(1)}_t(z)\mu(dt,dz)\\
  Y^{(1)}_T&=F,
 \end{cases}
\end{align}
It is easy to check that, under the assumptions $(C1)-(C5)$, for any fixed input $(Y^{(0)},Z^{(0)})\in L^2_{ad}(\GG)\times\Ii$, \eqref{RelaxedmfBSDE} satisfies the conditions of \cite[Theorem 4.5]{DiS} that yields existence and uniqueness of the solution $(Y^{(1)},Z^{(1)})$  in $\Ss^\GG_2 \times\Ii$. 
Remark that the cited result relies on the stochastic integral representation theorem for the martingale random field $\mu$ under the filtration $\GG$, see \cite[Theorem 3.3]{DiS}.

\vspace{2mm}
Define the mapping 
\begin{equation}
\label{psi}
\Psi: L^2_{ad}(\GG) \times\Ii\longrightarrow L^2_{ad}(\GG) \times\Ii
\end{equation}
$$(Y^{(0)},Z^{(0)})\longmapsto(Y^{(1)},Z^{(1)})$$
where $(Y^{(1)},Z^{(1)})$ is solution to \eqref{RelaxedmfBSDE}. 
If $\Psi$ is a contraction on $L^2_{ad}(\GG) \times\Ii$, then there exists a unique point in $(Y,Z) \in L^2_{ad}(\GG) \times\Ii$ such that $\Psi(Y,Z) = (Y,Z)$, which necessarily belongs to $\Ss^\GG_2\times\Ii$, as discussed above.
Furthermore the fixed point $(Y,Z)$ corresponds to the solution of the original equation \eqref{GeneralmfBSDE}.
Thus, we show that $\Psi$ has a unique fixed point by standard arguments via the Banach's fixed point theorem. This follows standard arguments. The details are in the Appendix.
\end{proof}

\vspace{2mm}
In the case of a linear mean-field BSDE, the set of assumptions guaranteeing existence can be detailed differently.
\begin{corollary}\label{CorollaryGeneralmfBSDE}
 Consider the case of the mean-field BSDE \eqref{GeneralmfBSDE} where $h:[0,T]\times\R^2\times\R^2\times\ \big(L^2(\delta_0+\nu)\big)^2\times \Omega^2 \longrightarrow\R$ has linear form:
\begin{align}\label{driverGeneralmfBSDE}
\begin{split}
 h(t,l,l',y,y',z(\cdot),z'(\cdot),\tilde\omega,\omega)&=A_t(\tilde\omega,\omega)+B_t(\omega)y+C_t(\tilde\omega,\omega)y'\\
 &\quad+D_t(0,\omega)z(0)l^{(1)}+E_t(0,\tilde\omega,\omega)z'(0)(l^{(1)})'\\
 &\quad+\int\limits_{\R_0}D_t(\xi,\omega)z(\xi)l^{(2)}+E_t(\xi,\tilde\omega,\omega)z'(\xi)(l^{(2)})'\nu(d\xi),
\end{split}
\end{align}
where $l=(l^{(1)},l^{(2)})$, $l'=((l^{(1)})',(l^{(2)})')$. 
Assume
 \begin{itemize}
  \item[(C1')] $F\in L^2(\Omega,\F,P)$ $\F_T$-measurable
  \item[(C2')] $A_\cdot,\,C_\cdot,\,E_\cdot(\xi)$ are $\F\otimes\GG$-adapted and $B_\cdot,\,D_\cdot(\xi)$ are $\GG$-adapted for all $\xi\in\R$.
  \item[(C3')] $B_\cdot,\,C_\cdot,\,D_\cdot(0)\sqrt{\lambda^B_\cdot},\,E_\cdot(0)\sqrt{(\lambda^B)'_\cdot},\int_{\R_0}|D_\cdot(\xi)|^2\lambda^H_\cdot\nu(d\xi)$ and $\int_{\R_0}|E_\cdot(\xi)|^2(\lambda^H)'_\cdot\nu(d\xi)$ are bounded.
  \item[(C4')] $A\in L^2_{pred}(\F\otimes\GG)$.
  \item[(C5')] $E'[A_\cdot+C_\cdot (Y^{(0)}_\cdot)'+E_\cdot(0)(Z^{(0)}_\cdot(0))'+\int_{\R_0}E_\cdot(\xi)(Z^{(0)}_\cdot(\xi))'\nu(d\xi)]\in L^2_{pred}(\GG)$ for all $(Y^{(0)},Z^{(0)})\in L^2_{ad}(\GG) \times\Ii$.
 \end{itemize}
 Then there exists a solution in $ \Ss^\GG_2 \times \Ii$ to the linear mean-filed BSDE.
\end{corollary}

\section{The mean-field stochastic control problem}

Let us consider the controlled stochastic process described by the following mean-field SDE: 
\begin{align}\label{ControlledDynamics}
 X^u_t=x+\int\limits_0^tb(s,\lambda_s,X^u_{s-},E[X^u_s],u_s)ds+\int\limits_0^t\int\limits_\R\kappa(s,z,\lambda_s,X^u_{s-},E[X^u_s],u_s)\mu(ds,dz),
\end{align}
where $u=(u_t)_{t\in[0,T]}$ denotes the control variable. Here,
\begin{align*}
 b&:[0,T]\times\R_+\times\R\times\R\times\R\times\Omega\longmapsto\R\\
 \kappa&:[0,T]\times\R\times\R_+\times\R\times\R\times\R\times\Omega\longmapsto\R.
\end{align*}
The dynamics \eqref{ControlledDynamics} are a special case of \eqref{mfSDE}.
Hereafter we reformulate and specify the assumptions $(E1)-(E3)$ to fit the present study.
From now on we shall assume the following conditions on the coefficients $b$ and $\kappa$ to hold. \begin{assumptions}\label{MainAssumptionsForwardDynamics}\hspace{10cm}
 \begin{itemize}
  \item[(E1')] $b$ and $\kappa$ can be decomposed as follows:
  \begin{align*}
   b(s,\lambda,x,y,u,\omega)&=b_0(\omega,s,\lambda)\cdot b_1(s,\lambda,x,y,u)+b_2(\omega,s,\lambda)\\
   \kappa(s,z,\lambda,x,y,u,\omega)&=\kappa_0(\omega,s,z,\lambda)\cdot\kappa_1(s,z,\lambda,x,y,u)+\kappa_2(\omega,s,z,\lambda),
  \end{align*}
 where $b_0$, $b_2$, $\kappa_0$, $\kappa_2$ are such that for $i=0,2$
 \begin{align*}
  &(\omega,s,z)\longmapsto b_i(\omega,s,\lambda_s(\omega)),\,(\omega,s,z)\longmapsto\kappa_i(\omega,s,z,\lambda_s(\omega))
 \end{align*}
 are $\FF$-predictable and $b_1$ and $\kappa_1$ are $C^1$ in $(s,z,\lambda,x,y,u)$.
  \item[(E2')] There exist the deterministic constants $0\leq K,L<\infty$ such that for $\partial_i b$ and $\partial_i\kappa$, $i=x,y,u$, the following boundedness and Lipschitzianity conditions hold $\text{Leb}([0,T]\times\R^3)\otimes P$-a.e. 
  \begin{align}
    &|\partial_i b(s,\lambda_s(\omega),x,y,u,\omega)|+\lVert\partial_i \kappa(s,\cdot,\lambda_s(\omega),x,y,u,\omega)\rVert_{\lambda_s}<K \label{boundedness}\\
    &|\partial_i b(s,\lambda,x_1,y_1,u_1,\omega)-\partial_i b(s,\lambda,x_2,y_2,u_2,\omega)|\leq L(|x_1-x_2|+|y_1-y_2|+|u_1-u_2|)\\
    &\lVert\partial_i\kappa(s,\lambda_s(\omega),x_1,y_1,u_1,\omega)-\partial_i\kappa(s,\lambda_s(\omega),x_2,y_2,u_2,\omega)\rVert_{\lambda_s} \notag
    \\
    &\hspace{6cm} \leq L(|x_1-x_2|+|y_1-y_2|+|u_1-u_2|)
   \end{align}
  \item[(E3')] $E\Big[\int_0^T|b(s,\lambda_s,0,0,0)|^2+\lVert\kappa(s,\cdot,\lambda_s,0,0,0)\rVert^2_{\lambda_s}ds\Big]<\infty$.
 \end{itemize}
\end{assumptions}

\vspace{2mm}
We introduce the space
 $\Hh^{\FF}$ of $\FF$-predictable processes in such that $\Vert Y\Vert_{\Ss_2}^2:= E[\sup_{s\leq T}|Y_s|^2]<\infty$.

\begin{lemma}\label{ExistenceAndUniquenessSDEForOurOptimizationProblem}
 Let $u\in\Hh^\FF$. Then the SDE \eqref{ControlledDynamics} has a unique solution in 
 $\Ss^\FF_2$.
\end{lemma}

\begin{proof}
 Define the random functions $b^{(\lambda,u)}:[0,T]\times\R\times M_2(\R)\times\Omega\rightarrow\R$, $\kappa^{(\lambda,u)}:[0,T]\times\R\times\R\times M_2(\R)\times\Omega\rightarrow\R$
 \begin{align*}
  b^{(\lambda,u)}(s,x,\Yy,\omega)&:=b(s,\lambda_s(\omega),x,\langle\id,\Yy\rangle,u_s(\omega))\\
  \kappa^{(\lambda,u)}(s,z,x,\Yy,\omega)&=\kappa(s,z,\lambda_s(\omega),x,\langle\id,\Yy\rangle,u_s(\omega)),
 \end{align*}
where $\langle\alpha,\Yy\rangle=\int_\R\alpha(a)\Yy(da)$. We verify that assumptions $(E1)-(E3)$ hold and apply Theorem \ref{ExistenceAndUniquenessmfSDE} to conclude. 
Observe that the particular structure of $b$ and $\kappa$ given in $(E1')$ implies $(E1)$. 
As for $(E2)$, we check the Lipschitzianity for $\kappa^{(\lambda,u)}$ only as the same argument can be applied to $b^{(\lambda,u)}$. 
By condition $(E1')$, the function $(x,y)\longmapsto\kappa(s,z,\lambda_s(\omega),x,y,u)$ is $C^1$. Then applying the generalisation of the mean value theorem for functions in several variables, there exists $\alpha=\alpha(\omega)\in[0,1]$, $\omega\in \Omega$, such that
\begin{align*}
 &\kappa(s,\cdot,\lambda_s(\omega),x_1,\langle\id,\Yy_1\rangle,u_s(\omega))-\kappa(s,\cdot,\lambda_s(\omega),x_2,\langle\id,\Yy_2\rangle,u_s(\omega))\\
 &=\Big\langle\nabla_{x,y}\kappa(s,\cdot,\lambda_s(\omega),\alpha(\omega) x_1+(1-\alpha(\omega))x_2,\langle\id,\alpha(\omega) \Yy_1+(1-\alpha(\omega))\Yy_2\rangle,u_s(\omega)),\begin{pmatrix}
            x_1-x_2\\                                                                                                                                                                                                                                    
            \langle\id,\Yy_1-\Yy_2\rangle                                                                                                                                                                                                                                   \end{pmatrix}\Big\rangle.
\end{align*}
This, together with Cauchy-Schwarz's inequality and the definition of $\lVert\cdot\rVert_{\lambda_s}$ yields
\begin{align*}
&\lVert\kappa^{(\lambda,u)}(s,z,x_1,\Yy_1,\omega)-\kappa^{(\lambda,u)}(s,z,x_2,\Yy_2,\omega)\rVert_{\lambda_s}\\
&\leq\lVert\partial_x\kappa(s,\cdot,\lambda_s(\omega),\tilde x(\omega),\langle\id,\tilde \Yy(\omega)\rangle,u_s(\omega))\rVert_{\lambda_s}\cdot|x_1-x_2|\\
&\quad+\lVert\partial_y\kappa(s,\cdot,\lambda_s(\omega),\tilde x(\omega),\langle\id,\tilde \Yy(\omega)\rangle,u_s(\omega))\rVert_{\lambda_s}\cdot|\langle\id,\Yy_1-\Yy_2\rangle|,
\end{align*}
where $\tilde x(\omega)=\alpha(\omega) x_1+(1-\alpha(\omega))x_2$ and $\tilde \Yy(\omega)=\alpha(\omega) \Yy_1+(1-\alpha(\omega))\Yy_2$. Moreover, the boundedness \eqref{boundedness} of the partial derivatives from $(E2')$ implies
\begin{align*}
\lVert\kappa^{(\lambda,u)}(s,z,x_1,\Yy_1,\omega)-\kappa^{(\lambda,u)}(s,z,x_2,\Yy_2,\omega)\rVert_{\lambda_s}\leq K(|x_1-x_2|+|\langle\id,\Yy_1-\Yy_2\rangle|).
\end{align*}
The Lipschitzianity of the identity and Kantorovich-Rubinstein's theorem give $(E2)$:
\begin{align*}
\lVert\kappa^{(\lambda,u)}(s,z,x_1,\Yy_1,\omega)-\kappa^{(\lambda,u)}(s,z,x_2,\Yy_2,\omega)\rVert_{\lambda_s}&\leq K_1(|x_1-x_2|+d_{\R}(\Yy_1,\Yy_2)).
\end{align*}
Finally, $(E3')$ and the $(E2)$ just proved imply $(E3)$.
\end{proof}

\vspace{3mm}
In the sequel we study the optimal control problem
\begin{equation} \label{OP}
J(\hat u) = \sup_{u \in \mathcal{A}} J(u)
\end{equation}
with objective functional
\vspace{-3mm}
\begin{align}\label{Objective}
 J(u):=E\Big[\int\limits_0^Tf(s,\lambda_s,X^u_{s-},E[\varphi(X^u_{s})],u_s)ds+g(X^u_{T},E[\chi(X^u_T)])\Big]
\end{align}
for the dynamics \eqref{ControlledDynamics} and on a class of admissible controls $\mathcal{A}$ characterised below.

The objective function $J$ is subject to the following assumptions.
\begin{assumptions}\label{MainAssumptions}\hspace{10cm}
 \begin{itemize}
  \item[(O1)] For all $(s,z,\omega)\in[0,T]\times\R\times\Omega$:
  \begin{align*}
   &(x,y,u)\longmapsto f(s,\lambda_s(\omega),x,y,u)\in C^1(\R^3)\\
   &(x,y)\longmapsto g(x,y)\in C^1(\R^2)\\
   &\varphi,\,\chi\in C^1(\R).
  \end{align*}
  \item[(O2)] $g$, $\varphi$, $\chi$ are concave.
  \item[(O3)] $\partial_x\varphi$ and $\partial_x\chi$ are Lipschitz.
  \item[(O4)] It holds
   \begin{itemize}
   \item either $\varphi$ is affine or $\partial_yf(s,\lambda_s(\omega),x,y,u)\geq0$ for all $(s,z,x,y,u,\omega)\in[0,T]\times\R\times\R\times\R\times\R\times\Omega$.
   \item either $\chi$ is affine or $\partial_yg(x,y)\geq0$ for all $(x,y)\in\R\times\R$.
   \end{itemize}
  \item[(O5)] $g$ is such that for all $X\in L^2(\Omega,\F,P)$ and all $y\in\R$:
  \begin{align*}
   &\partial_xg(X,y)\in L^2(\Omega,\F,P)\text{ and }\partial_yg(X,y)\in L^1(\Omega,\F,P).
  \end{align*}
 \end{itemize}
\end{assumptions}

Hereafter we characterise the admissible strategies.
\begin{definition}\label{Admissibility}
 Let $U\subseteq\R$ be a convex set. A stochastic process $u\in\Hh^{\FF}$ with values in $U$ is called an admissible strategy if the following conditions are satisfied
 \begin{itemize}
  \item[(A1)] The objective $J(u)$ is well defined for $u$, i.e.
  \begin{align*}
   s\longmapsto f(s,\lambda_s,X^u_{s-},E[\varphi(X^u_{s})],u_s)\in L^1(\Omega\times[0,T],\F\otimes\B([0,T]),P\otimes\text{Leb})
  \end{align*}
  and
  \vspace{-3mm}
  \begin{align*}
   g(X^u_{T},E[\chi(X^u_{T})])\in L^1(\Omega,\F,P).
  \end{align*}
  \item[(A2)] For $i=x,y,u$, the stochastic processes 
  $\partial_if(s,\lambda_s,X^u_{s-},E[\varphi(X^u_{s})],u_s)$,
  $(s,\omega) \in [0,T]\times \Omega$,
  are elements of $L^2([0,T]\times\Omega,\B([0,T])\otimes\F, \text{Leb}\otimes P)$.
  For $i=x,y$, the random variables
  $ \partial_i g(X^u_{T},E[\chi(X^u_{T})])$, $\omega\in\Omega$ 
  belong to $L^2(\Omega,\F,P)$.
  \end{itemize}
 The set of admissible strategies is denoted by $\A$.
\end{definition}

The presence of the mean-field terms makes the optimal control problem \eqref{OP} inhomogeneous in the sense that it does not satisfy the Bellman principle. 
We study the problem \eqref{OP} via the stochastic maximum principle and we suggest a sufficient and a necessary result. 
For these we shall work with the Hamiltonian function in which the solution of the adjoint equation appears. 
In the context of this paper the adjoint equation is a mean-field BSDE driven by time changed L\'evy noises.

In order to make things more readable, we introduce the following short-hand notation
\begin{align*}
 &\hat X_t:=X^{\hat u}_t,\\
 &\hat\varphi_t:=\varphi(\hat X_{t}),\,\hat\chi_T:=\chi(\hat X_T)\\
 &b_t:=b(t,\lambda_t,X^u_{t-},E[X^u_{t}],u_t),\,\kappa_t(\cdot),\,f_t,\,g_T\text{ accordingly},\\
 &\hat b_t:=b(t,\lambda_t,\hat X_{t-},E[\hat X_t],\hat u_t),\,\hat\kappa_t(\cdot),\,\hat f_t,\,\hat g_T\text{ accordingly}.
 \end{align*}
The adjoint equation has the form below:
\begin{align}
\begin{split}\label{adjointeq}
 d\hat p_t&=-\Big\{\partial_x\hat f_t+\partial_x\hat b_t\cdot\hat p_t+\partial_x\hat\kappa_t(0)\hat q_t(0)\lambda^B_t+\int\limits_{\R_0}\partial_x\hat\kappa_t(z)\hat q_t(z)\lambda^H_t\nu(dz)+E[\partial_y\hat f_t]\partial_x\hat\varphi_t\\
 &\quad+E[\partial_y\hat b_t\cdot\hat p_t]+E\Big[\partial_y\hat\kappa_t(0)\hat q_t(0)\lambda^B_t+\int\limits_{\R_0}\partial_y\hat\kappa_t(z)\hat q_t(z)\lambda^H_t\nu(dz)\Big]\Big\}dt\\
 &\quad+\int\limits_\R\hat q_t(z)\mu(dt,dz)\\
 \hat p_T&=\partial_x\hat g_T+E[\partial_y\hat g_T]\partial_x\hat\chi_T
\end{split}
\end{align}
\begin{remark}
 It follows again from Doob's regularisation theorem (Theorem 6.27 in \cite{K}) that we can replace the $\hat p_{t-}$ by $\hat p_t$ inside any integral w.r.t. $dt$ if either of the two versions of the BSDE has a solution. The same applies to $X_{t-}$ and $X_t$. We will apply this regularly in the next sections without additional notice.
\end{remark}

To make sense of a solution to \eqref{adjointeq} we embed the equation in the theory of Section 4.
Notice that there the analysis is carried through under filtration $\GG$. 
Indeed it is under $\GG$ that an appropriate stochastic integral representation theorem is provided.
However, the stochastic control problem \eqref{OP} we are facing is given under the information flow $\FF$, which is more reasonable from a modelling perspective.
We shall deal with this form of "partial" information in the sequel.

\begin{lemma}
 Let $\hat u\in\A$. Then the adjoint equation \eqref{adjointeq} has a unique solution in $S^\GG_2\times\mathcal{I}$.
 \end{lemma}

\begin{proof}
In the notation of Section 4, by the relationship \eqref{correspondence}, we can rewrite the adjoint equation \eqref{adjointeq} as
\begin{align*}
 \begin{split}
  d\hat p_t&=E'\Big[A_t+B_t\cdot\hat p_{t}+C_t\cdot\hat p_{t}'+D_t(0)\hat q_t(0)\lambda^B_t+E_t(0)\hat q_t'(0)(\lambda^B_t)'\\
  &\quad+\int\limits_{\R_0}D_t(z)\hat q_t(z)\lambda^H_t+E_t(z)\hat q_t'(z)(\lambda^H_t)'\nu(dz)\Big]dt\\
  &\quad+\int\limits_\R\hat q_t(z)\mu(dt,dz),\\
  \hat p_T&=F
 \end{split} 
\end{align*}
where
\vspace{-2mm}
\begin{align*}
 A_t&=\partial_x\hat f_t+(\partial_y\hat f_t)'\partial_x\hat\varphi_t \\
 B_t &=\partial_x\hat b_t \\
 C_t&=(\partial_y\hat b_t)' \\
 D_t(\cdot)&=\partial_x\hat\kappa_t(\cdot) \\
 E_t(\cdot)&=(\partial_y\hat\kappa_t(\cdot))' \\
 F&=\partial_x\hat g_T+E[\partial_y\hat g_T]\partial_x\hat\chi_T. 
\end{align*}
Being an equation of linear type we apply Corollary \ref{CorollaryGeneralmfBSDE} after verifying the conditions required. This can be easily done and we omit the details.
\end{proof}


\subsection{A sufficient stochastic maximum principle}
Let us now define the Hamiltonian function
\begin{align}
 H(t,\lambda_t,x,y_1,y_2,u,p,q)&:=f(t,\lambda_t,x,y_1,u)+b(t,\lambda_t,x,y_2,u)\cdot p\nonumber\\
 &\quad+\kappa(t,0,\lambda_t,x,y_2,u)q(0)\lambda^B_t\\
 &\quad+\int\limits_{\R_0}\kappa(t,z,\lambda_t,x,y_2,u)q(z)\lambda^H_t\nu(dz).\nonumber
\end{align}
We introduce an $\FF$-Hamiltonian given by
\begin{align}
 H^\FF(t,\lambda_t,x,y_1,y_2,u,\hat p_{t-},\hat q_t)&:=E[H(t,\lambda_t,x,y_1,y_2,u,\hat p_{t-},\hat q_t)|\F_t]\\
 &=f(t,\lambda_t,x,y_1,u)+b(t,\lambda_t,x,y_2,u)E[\hat p_{t-}|\F_t]\nonumber\\
 &\quad+\kappa(t,0,\lambda_t,x,y_2,u)E[\hat q_t(0)|\F_t]\lambda^B_t\nonumber\\
 &\quad+\int\limits_{\R_0}\kappa(t,z,\lambda_t,x,y_2,u)E[\hat q_t(z)|\F_t]\lambda^H_t\nu(dz),\nonumber
\end{align}
where $(\hat p,\hat q)$ is the solution to the adjoint equation \eqref{adjointeq}. 
As anticipated earlier we deal with a form of partial information given by $\FF$ when compared with $\GG$. Note that $\GG$ includes the information of the whole evolution of the time change process $\lambda$, hence not feasible from a modelling perspective. For this we adopt techniques from \cite{DiS}.
Hereafter we formulate a sufficient maximum principle in the framework of Assumptions \ref{MainAssumptionsForwardDynamics} and \ref{MainAssumptions}. 

\begin{theorem}\label{SufficientMaxPrinciple}
 Let $\hat u\in\A$ and $(\hat p, \hat q)$ be the solution of the mean-field BSDE \eqref{adjointeq}. If the function
 \begin{align}\label{SufficientMaxPrincipleEq1}
  h_t(x,y_1,y_2)&:=\sup_{v\in U}H^\FF(t,\lambda_t,x,y_1,y_2,v,\hat p_{t-},\hat q_t)
 \end{align}
 exists for all $t\in[0,T]$, $P$-a.s., and is concave in $(x,y_1,y_2)$ and if furthermore
 \begin{align}\label{SufficientMaxPrincipleEq2}
  &H^\FF(t,\lambda_t,X^{\hat u}_{t-},E[\varphi(X^{\hat u}_{t})],E[X^{\hat u}_{t}],\hat u_t,\hat p_{t-},\hat q_t)=h_t(X^{\hat u}_{t-},E[\varphi(X^{\hat u}_{t})],E[X^{\hat u}_{t}]),
 \end{align}
 then $\hat u$ is an optimal control.
\end{theorem}

\begin{proof}
 Let $u\in\A$. Define $X_t:=X^u_t$, $\hat X_t:= X^{\hat u}_t$ and use the short-hand notation introduced in the previous section. 
 We shall prove that, for any $u\in \mathcal{A}$,  
 $$
 J(\hat u)-J(u)=E[\hat g_T- g_T]+E\Big[\int\limits_0^T\hat f_s- f_s ds\Big].
 $$ 
First of all we observe that
 \begin{align*}
  E[\hat g_T- g_T]
  &\geq E[\partial_x\hat g_T\cdot(\hat X_T-X_T)+\partial_y\hat g_T\cdot E[\partial_x\hat\chi_T\cdot(\hat X_T-X_T)]]\\
 & \geq E[\partial_x\hat g_T\cdot(\hat X_T-X_T)+E[\partial_y\hat g_T]\cdot \partial_x\hat\chi_T\cdot(\hat X_T-X_T)]\\
  &=E[\hat p_T\cdot(\hat X_T-X_T)-\hat p_0\cdot(\hat X_0-X_0)]
 \end{align*}
 where we have used the product rule, the form of the technical condition of  \eqref{adjointeq}, $\hat X_0=X_0=x$, and the observation that, for any two random variables $X$ and $Y$,
 \begin{align}\label{StandardTrickExpectations}
  E[E[X]\cdot Y]=E[X]\cdot E[Y]=E[X\cdot E[Y]].
 \end{align}
 Applying the It\^o's formula with the dynamics \eqref{ControlledDynamics} and \eqref{adjointeq} we have
 \begin{align*}
  E[\hat g_T- g_T]&\geq E[\hat p_T\cdot(\hat X_T-X_T)-\hat p_0\cdot(\hat X_0-X_0)]\\
  &=E\Big[\int\limits_0^T\hat p_{s-}(\hat b_s-b_s)ds+\int\limits_0^T\int\limits_\R\hat p_{s-}(\hat\kappa_s(z)-\kappa_s(z))\mu(ds,dz)\\
  &\quad+\int\limits_0^T-(\hat X_{s-}-X_{s-})\Big\{\partial_x\hat f_s+\partial_x\hat b_s\cdot\hat p_s+\partial_x\hat\kappa_s(0)\hat q_s(0)\lambda^B_s\\
  &\quad+\int\limits_{\R_0}\partial_x\hat\kappa_s(z)\hat q_s(z)\lambda^H_s\nu(dz)+E[\partial_y\hat f_s]\partial_x\hat\varphi_s+E[\partial_y\hat b_s\cdot\hat p_s]\\
  &\quad+E\Big[\partial_y\hat\kappa_s(0)\hat q_s(0)\lambda^B_s+\int\limits_{\R_0}\partial_y\hat\kappa_s(z)\hat q_s(z)\lambda^H_s\nu(dz)\Big]\Big\}ds\\
  &\quad+\int\limits_0^T\int\limits_\R(\hat X_{s-}-X_{s-})\hat q_s(z)\mu(ds,dz)\\
  &\quad+\int\limits_0^T(\hat\kappa_s(0)-\kappa_s(0))\hat q_s(0)\lambda^B_s+\int\limits_\R(\hat\kappa_s(z)-\kappa_s(z))\hat q_s(z)\lambda^H_s\nu(dz)\ds\Big]
 \end{align*}
Recall that $\mu$ is a martingale random and It\^o calculus rules apply. Then using \eqref{StandardTrickExpectations} we obtain
 \begin{align*}
  &E[\hat g_T- g_T]\\
  &\geq E\Big[\int\limits_0^T\hat p_{s-}(\hat b_s-b_s)ds+\int\limits_0^T(\hat\kappa_s(0)-\kappa_s(0))\hat q_s(0)\lambda^B_s+\int\limits_\R(\hat\kappa_s(z)-\kappa_s(z))\hat q_s(z)\lambda^H_s\nu(dz)\ds\\
  &\quad+\int\limits_0^T-(\hat X_{s-}-X_{s-})\Big\{\partial_x\hat f_s+\partial_x\hat b_s\cdot\hat p_{s-}+\partial_x\hat\kappa_s(0)\hat q_s(0)\lambda^B_s+\int\limits_{\R_0}\partial_x\hat\kappa_s(z)\hat q_s(z)\lambda^H_s\nu(dz)\Big\}ds\\
  &\quad+\int\limits_0^T-\Big\{\partial_y\hat f_s\cdot E[\partial_x\hat\varphi_s\cdot(\hat X_s-X_s)]+\partial_y\hat b_s\cdot\hat p_{s-}\cdot E[\hat X_s-X_s]\\
  &\quad+\Big(\partial_y\hat\kappa_s(0)\hat q_s(0)\lambda^B_s+\int\limits_{\R_0}\partial_y\hat\kappa_s(z)\hat q_s(z)\lambda^H_s\nu(dz)\Big)\cdot E[\hat X_s-X_s]\Big\}ds\Big].
 \end{align*}
 Since $u,\hat u\in\A$, the terms that are $\GG$-adapted, but not necessarily $\FF$-adapted, are $\hat p$ and $\hat q$. So Fubini's theorem and the tower property yield
 \begin{align}
  &E[\hat g_T- g_T] \notag\\
  &\geq E\Big[\int\limits_0^T\Big\{E[\hat p_{s-}|\F_s](\hat b_s-b_s)+(\hat\kappa_s(0)-\kappa_s(0))E[\hat q_s(0)|\F_s]\lambda^B_s+\int\limits_\R(\hat\kappa_s(z)-\kappa_s(z))E[\hat q_s(z)|\F_s]\lambda^H_s\nu(dz)\Big\} \notag\\
  &\quad-(\hat X_{s-}-X_{s-})\Big\{\partial_x\hat f_s+\partial_x\hat b_s\cdot E[\hat p_{s-}|\F_s]+\partial_x\hat\kappa_s(0)E[\hat q_s(0)|\F_s]\lambda^B_s+\int\limits_{\R_0}\partial_x\hat\kappa_s(z)E[\hat q_s(z)|\F_s]\lambda^H_s\nu(dz)\Big\} \notag\\
  &\quad-\Big\{\partial_y\hat f_s\cdot E[\partial_x\hat\varphi_s\cdot(\hat X_s-X_s)]+\partial_y\hat b_s\cdot E[\hat p_{s-}|\F_s]\cdot E[\hat X_s-X_s] \notag\\
  &\quad+\Big(\partial_y\hat\kappa_s(0)E[\hat q_s(0)|\F_s]\lambda^B_s+\int\limits_{\R_0}\partial_y\hat\kappa_s(z)E[\hat q_s(z)|\F_s]\lambda^H_s\nu(dz)\Big)\cdot E[\hat X_s-X_s]\Big\}ds\Big].\label{Subresult1}
 \end{align}
Observe that
 \begin{align*}
  &E[\hat p_{s-}|\F_s](\hat b_s-b_s)+(\hat\kappa_s(0)-\kappa_s(0))E[\hat q_s(0)|\F_s]\lambda^B_s+\int\limits_\R(\hat\kappa_s(z)-\kappa_s(z))E[\hat q_s(z)|\F_s]\lambda^H_s\nu(dz)\\
  &=\Big(H^\FF(s,\lambda_s,\hat X_{s-},E[\varphi(\hat X_{s})],E[\hat X_{s}],\hat u_s,\hat p_{s-},\hat q_s)-H^\FF(s,\lambda_s,X_{s-},E[\varphi(X_{s})],E[X_{s}],u_s,\hat p_{s-},\hat q_s)\Big)\\
  &\quad-\Big(\hat f_s-f_s\Big).
 \end{align*}
 By \eqref{SufficientMaxPrincipleEq1} and \eqref{SufficientMaxPrincipleEq2} we have that, for all $(x,y_1,y_2)\in\R^3$,
 \begin{align}\label{Hilfsgl1}
 \begin{split}
  &H^\FF(s,\lambda_s,\hat X_{s-},E[\hat\varphi_s],E[\hat X_s],\hat u_s,\hat p_{s-},\hat q_s)-H^\FF(s,\lambda_s,x,y_1,y_2,u_s,\hat p_{s-},\hat q_s)\\
  &\geq h^\FF_s(\hat X_{s-},E[\varphi(\hat X_s)],E[\hat X_s])-h^\FF_s(x,y_1,y_2),
 \end{split}
 \end{align}
 and thus the two relationships above give
 \begin{align}\label{Subresult2}
  &E[\hat p_{s-}|\F_s](\hat b_s-b_s)+(\hat\kappa_s(0)-\kappa_s(0))E[\hat q_s(0)|\F_s]\lambda^B_s+\int\limits_\R(\hat\kappa_s(z)-\kappa_s(z))E[\hat q_s(z)|\F_s]\lambda^H_s\nu(dz) \notag\\
  &\geq\Big(h^\FF_s(\hat X_{s-},E[\varphi(\hat X_s)],E[\hat X_s])-h^\FF_s(X_{s-},E[\varphi(X_s)],E[X_s])\Big)-\Big(\hat f_s-f_s\Big).
 \end{align}
 Now, by the concavity of $h^\FF_s$ and a separating hyperplane argument, there exists a vector $a\in\R^3$ such that for all $(x,y_1,y_2)\in\R^3$:
 \begin{align}\label{Hilfsgl2}
  h^\FF_s(\hat X_{s-},E[\varphi(\hat X_s)],E[\hat X_s])-h^\FF_s(x,y_1,y_2)-\Big\langle a,\begin{pmatrix}
               \hat X_{s-}-x\\
               E[\varphi(\hat X_s)]-y_1\\
               E[\hat X_s]-y_2
              \end{pmatrix}\Big\rangle\geq0
 \end{align}
 Define
 \begin{align*}
  \rho(x,y_1,y_2)&:=H^\FF(s,\lambda_s,\hat X_{s-},E[\varphi(\hat X_s)],E[\hat X_s],\hat u_s,\hat p_{s-},\hat q_s)\\
  &\quad-H^\FF(s,\lambda_s,x,y_1,y_2,\hat u_s,\hat p_{s-},\hat q_s)-\Big\langle a,\begin{pmatrix}
               \hat X_{s-}-x\\
               E[\varphi(\hat X_s)]-y_1\\
               E[\hat X_s]-y_2
              \end{pmatrix}\Big\rangle
 \end{align*}
 By \eqref{Hilfsgl1}, it holds $\rho(x,y_1,y_2)\geq0$. On the other hand, we have
 \begin{align*}
  \rho(\hat X_{s-},E[\varphi(\hat X_s)],E[\hat X_s])=0,
 \end{align*}
 i.e. $\rho$ obtains a maximum in $(\hat X_{s-},E[\varphi(\hat X_s)],E[\hat X_s])$. 
 Since $\rho$ is $C^1$ we have
 \begin{align*}
  0&=\nabla\rho(\hat X_{s-},E[\varphi(\hat X_s)],E[\hat X_s])\\
  &=\nabla H^\FF(s,\lambda_s,x,y_1,y_2,u_s,\hat p_{s-},\hat q_s)\Big|_{\substack{x=\hat X_{s-},\,y_1=E[\varphi(\hat X_s)],\\ y_2=E[\hat X_s]}}-a,
 \end{align*}
 where $\nabla$ denotes the gradient w.r.t. $(x,y_1,y_2)$. Plugging this into \eqref{Hilfsgl2} finally we obtain
 \begin{align}\label{Subresult3}
  \begin{split}
   0\leq&h^\FF_s(\hat X_{s-},E[\varphi(\hat X_s)],E[\hat X_s])-h^\FF_s(X_{s-},E[\varphi(X_s)],E[X_s])\\
   &\quad-\Big\langle \nabla H^\FF(s,\lambda_s,\hat X_{s-},E[\hat\varphi_s],E[\hat X_s],u_s,\hat p_{s-},\hat q_s),\begin{pmatrix}
               \hat X_{s-}-X_{s-}\\
               E[\varphi(\hat X_s)-\varphi(X_s)]\\
               E[\hat X_s-X_s]
              \end{pmatrix}\Big\rangle.
  \end{split}
 \end{align}
Since $\varphi$ is concave and either $\varphi$ is affine or $\partial_y\hat f_s\geq0$, it holds
 \begin{align*}
  &\Big\{\partial_x\hat f_s+\partial_x\hat b_s\cdot E[\hat p_{s-}|\F_s]+\partial_x\hat\kappa_s(0)E[\hat q_s(0)|\F_s]\lambda^B_s+\int\limits_{\R_0}\partial_x\hat\kappa_s(z)E[\hat q_s(z)|\F_s]\lambda^H_s\nu(dz)\Big\}(\hat X_{s-}-X_{s-})\\
  &\quad+\Big\{\partial_y\hat f_s\cdot E[\partial_x\hat\varphi_s\cdot(\hat X_s-X_s)]+\partial_y\hat b_s\cdot E[\hat p_{s-}|\F_s]\cdot E[\hat X_s-X_s]\\
  &\quad+\Big(\partial_y\hat\kappa_s(0)E[\hat q_s(0)|\F_s]\lambda^B_s+\int\limits_{\R_0}\partial_y\hat\kappa_s(z)E[\hat q_s(z)|\F_s]\lambda^H_s\nu(dz)\Big)\cdot E[\hat X_s-X_s]\Big\}\\  
  &\leq\Big\{\partial_x\hat f_s+\partial_x\hat b_s\cdot E[\hat p_{s-}|\F_s]+\partial_x\hat\kappa_s(0)E[\hat q_s(0)|\F_s]\lambda^B_s+\int\limits_{\R_0}\partial_x\hat\kappa_s(z)E[\hat q_s(z)|\F_s]\lambda^H_s\nu(dz)\Big\}(\hat X_{s-}-X_{s-})\\
  &\quad+\Big\{\partial_y\hat f_s\cdot E[\varphi_s(\hat X_s)-\varphi_s(X_s)]+\partial_y\hat b_s\cdot E[\hat p_{s-}|\F_s]\cdot E[\hat X_s-X_s]\\
  &\quad+\Big(\partial_y\hat\kappa_s(0)E[\hat q_s(0)|\F_s]\lambda^B_s+\int\limits_{\R_0}\partial_y\hat\kappa_s(z)E[\hat q_s(z)|\F_s]\lambda^H_s\nu(dz)\Big)\cdot E[\hat X_s-X_s]\Big\}\\
  &=\Big\langle \nabla H^\FF(s,\lambda_s,\hat X_{s-},E[\hat\varphi_s],E[\hat X_s],u_s,\hat p_{s-},\hat q_s),\begin{pmatrix}
               \hat X_{s-}-X_{s-}\\
               E[\varphi(\hat X_s)-\varphi(X_s)]\\
               E[\hat X_s-X_s]
              \end{pmatrix}\Big\rangle
 \end{align*}
 Plugging this last result, together with \eqref{Subresult2} and \eqref{Subresult3}, into \eqref{Subresult1}, we finally obtain
$$
  E[\hat g_T-g_T]\geq E\Big[-\int\limits_0^T\hat f_s-f_sds\Big],
$$
thus $J(\hat u)-J(u)\geq0$.
\end{proof}

\subsection{A necessary stochastic maximum principle}
For a necessary maximum principle we introduce additional specifications to the objective functional \eqref{Objective}.
\begin{assumptions}\label{AdditionalAssumptionsNecessMaxPrin}\hspace{10cm}
 \begin{itemize}
  \item[(O6)] The functions $f$ and $g$ have quadratic increments in $(x,y,u)$ uniformly in $(s,\lambda)$, i.e.
  \begin{align*}
   &|f(s,\lambda,x_1,y_1,u_1)-f(s,\lambda,x_2,y_2,u_2)|+|g(x_1,y_1)-g(x_2,y_2)|\\
   &\leq M(1+|x_1|+|x_2|+|y_1|+|y_2|+|u_1|+|u_2|)(|x_1-x_2|+|y_1-y_2|+|u_1-u_2|)
  \end{align*}
  \item[(O7)] Each of the functions
  \begin{align*}
   (x,y,u)&\longmapsto\partial_i f(s,\lambda,x,y,u),\,i=x,y,u\\
   (x,y)&\longmapsto\partial_i g(x,y),\,i=x,y
  \end{align*}
  is either Lipschitz (uniformly in $(s,\lambda)$ in the case of $f$) \textbf{or} independent of $s$ and $\lambda$ (automatically fulfilled by $g$) and bounded.
 \end{itemize}
\end{assumptions}
\begin{remark}\label{AlternativeA9}
 Note that, by differentiability of $f$ and $g$, assumption $(A2)$ is equivalent to
 \begin{align*}
  |\partial_i f(s,\lambda,x,y,u)|&\leq M(1+2|x|+2|y|+2|u|),\,i=x,y,u\\
  |\partial_i g(x,y)|&\leq M(1+2|x|+2|y|),\,i=x,y.
 \end{align*}
 This can be easily checked by the fundamental theorem of calculus.
\end{remark}

First we present some preparatory lammata. The arguments of the proofs are rather classical and the structure is similar to the one in \cite{AD}. Some of our conditions differ to fit the framework and the theory we presented in the earlier sections.
\begin{lemma}\label{DerivativeProcesswrtuPsiPhiAffine}
 Let $\hat u\in\A$ and $v\in\Hh^{\FF}$ such that for $\theta$ small enough, $\hat u+\theta v\in\A$.
 Then, the following mean-field SDE has a unique solution in $S^\FF_2$:
 \begin{align}
  dZ_t&=\{\partial_x\hat b_t\cdot Z_{t-}+\partial_y\hat b_tE[Z_t]+\partial_u\hat b_t\cdot v_t\}dt \notag\\
  &\quad+\int\limits_\R\{\partial_x\hat\kappa_t(z)\cdot Z_{t-}+\partial_y\hat\kappa_t(z)E[Z_t]+\partial_u\hat\kappa_t(z)\cdot v_t\}\mu(dt,dz)\notag\\
  Z_0&=0.\label{ZSDEaffine}
 \end{align}
 Moreover,
 \begin{align}\label{convergence}
  E\Big[\sup_{t\in[0,T]}\Big|\frac{X^{\hat u+\theta v}_t-X^{\hat u}_t}{\theta}-Z_t\Big|^2\Big]\rightarrow0,\text{ as }\theta\rightarrow0.
 \end{align}
\end{lemma}

\begin{proof}
Notice that \eqref{ZSDEaffine} we can be rewriten as 
\begin{align*}
  Z_t&=\int\limits_0^t\tilde b(s,Z_{s-},\Ll_{Z_s})ds+\int\limits_0^t\int\limits_\R\tilde\kappa(s,z,Z_{s-},\Ll_{Z_s})\mu(ds,dz),
 \end{align*}
 where
 \vspace{-2mm}
 \begin{align*}
  \tilde b(s,x,\Yy)&=\partial_x\hat b_s\cdot x+\partial_y\hat b_s\langle\id,\Yy\rangle+\partial_u\hat b_s\cdot v_s \\
  \tilde\kappa(s,z,x,\Yy)&=\partial_x\hat\kappa_s(z)\cdot x+\partial_y\hat\kappa_s(z)\langle\id,\Yy\rangle+\partial_u\hat\kappa_s(z)\cdot v_s
 \end{align*}
 To show existence of the solution we apply Theorem \ref{ExistenceAndUniquenessmfSDE} after having checked that $\tilde b$ and $\tilde\kappa$ satisfy conditions $(E1)-(E3)$.
 
 \vspace{2mm}
 Concerning $(E1)$, we make use of the special structure of $\tilde b$ and $\tilde\kappa$.
 In fact
  \begin{align*}
  \partial_i b(s,\lambda,x,y,u,\omega)&=b_0(s,\lambda,\omega)\partial_i b_1(s,\lambda,x,y,u) \\
  \partial_i\kappa(s,\lambda,x,y,u,\omega)&=\kappa_0(s,z,\lambda,\omega)\partial_i\kappa_1(s,z,\lambda,x,y,u) \quad\text{ for }i=x,y,u,
 \end{align*}
 where we recall condition $(E1')$ and we note that
 \begin{align*}
  (s,z,\omega)\longmapsto\begin{pmatrix}
                      s\\
                      \lambda_s(\omega)\\
                      X^{\hat u}_{s-}(\omega)\\
                      \langle\id,\Ll_{X^{\hat u}_s}\rangle\\
                      \hat u_s(\omega)
                     \end{pmatrix}
 \end{align*}
 is $\PF$-measurable, as $\lambda$ and $\hat u$ are $\FF$-predictable, $(X_{t-})_{t\in[0,T]}$ is $\FF$-adapted and c\`agl\`ad and $t\longmapsto\langle\id,\Ll_{X^{\hat u}_{t}}\rangle$ is deterministic. Then the processes
 \begin{align*}
  \partial_i b_1(s,\lambda_t(\omega),X^{\hat u}_{s-},\langle\id,\Ll_{X^{\hat u}_s}\rangle,u_s)\text{ and }\partial_i\kappa_1(s,z,\lambda_s(\omega),X^{\hat u}_{s-},\langle\id,\Ll_{X^{\hat u}_s}\rangle,u_s)
 \end{align*}
 are $\FF$-predictable. 
 Again by $(E1')$ the processes $b_0$ and $\kappa_0$ are $\PF$-measurable.
 Since $v\in\Hh^{\FF}$ is predictable, this implies that $\partial_x\hat b_s,\,\partial_y\hat b_s,\,\partial_u\hat b_s\cdot v_s,\,\partial_x\hat\kappa_s(\cdot),\,\partial_y\hat\kappa_s(\cdot),\,\partial_u\hat\kappa_s(\cdot)\cdot v_s$ are all predictable processes,then also the processes
 $ 
\tilde b(s,x,\Yy,\omega)$ and 
$\tilde\kappa(s,z,x,\Yy,\omega)$ are $\PF$-measurable and therefore $(E1)$ is fulfilled.
 
 \vspace{2mm}
 Observe that, by $(E2')$,
 \begin{align*}
   &|\tilde b(s,x_1,\Yy_1,\omega)-\tilde b(s,x_2,\Yy_2,\omega)|+\lVert\tilde\kappa(s,\cdot,x_1,\Yy_1,\omega)-\tilde\kappa(s,\cdot,x_2,\Yy_2,\omega)\rVert_{\lambda_s}\\
   &=|\partial_x\hat b_s(\omega)\cdot (x_1-x_2)+\partial_y\hat b_s(\omega)\langle\id,\Yy_1-\Yy_2\rangle|\\
   &\quad+\lVert\partial_x\hat\kappa_s(\cdot,\omega)\cdot (x_1-x_2)+\partial_y\hat\kappa_s(\cdot,\omega)\langle\id,\Yy_1-\Yy_2\rangle\rVert_{\lambda_s}\\
   &\leq4K(|x_1-x_2|+|\langle\id,\Yy_1-\Yy_2\rangle|)\leq C(|x_1-x_2|+d_{\R}(\Yy_1,\Yy_2)).
 \end{align*}
 Hence $(E2)$ holds. Finally consider 
$v\in\Hh^{\FF}$, then
 \begin{align*}
  E\Big[\int_0^T|\tilde b(s,0,\delta_0)|^2+\lVert\tilde\kappa(s,\cdot,0,\delta_0)\rVert_{\lambda_s}^2ds\Big]&=E\Big[\int_0^T|\partial_u\hat b_s\cdot v_s|^2+\lVert\partial_u\hat\kappa_s(\cdot)\cdot v_s\rVert_{\lambda_s}^2ds\Big]\\
  &2K^2E\Big[\int_0^T|v_s|^2ds\Big]<\infty,
 \end{align*}
which give (E3).

\vspace{2mm}
We now prove \eqref{convergence}. 
Fix now $\theta>0$ so that $\hat u+\theta v\in\A$. Since the mean-field SDE for $X^{\hat u+\theta v}$ has a unique solution in $\Ss^\FF_2$, then also the mean-field SDE for $Y^\theta_t:=\frac{X^{\hat u+\theta v}_t-X^{\hat u}_t}{\theta}-Z_t$ has a unique solution. 
Observe that $X^{\hat u+\theta v}_t=X^{\hat u}_t+\theta(Y^\theta_t+Z_t)$. Now define
 \begin{align*}
  b^{a}_t&:=b(t,\lambda_t,X^{\hat u}_{t-}+a(Y^\theta_{t-}+Z_{t-}),E[X^{\hat u}_t+a(Y^\theta_t+Z_t)],\hat u_t+a v_t),\\
  \kappa^{a}_t(z)&:=\kappa(t,z,\lambda_t,X^{\hat u}_{t-}+a(Y^\theta_{t-}+Z_{t-}),E[X^{\hat u}_t+a(Y^\theta_t+Z_t)],\hat u_t+a v_t).
 \end{align*}
 It holds $\hat b_t=b^{0}_t$ and
 \begin{align*}
  b^{\theta}_t-\hat b_t&=\int\limits_0^1\frac{d}{dl}b^{l\cdot\theta}_tdl=\int\limits_0^1\Big\langle\nabla b^{l\cdot\theta}_t,\begin{pmatrix}
 \theta(Y^\theta_{t-}+Z_{t-})\\
 \theta E[Y^\theta_t+Z_t]\\
\theta v_t
\end{pmatrix}\Big\rangle dl\\
  &=\theta\int\limits_0^1\partial_xb^{l\cdot\theta}_t\cdot(Y^\theta_{t-}+Z_{t-})+\partial_yb^{l\cdot\theta}_tE[Y^\theta_t+Z_t]+\partial_ub^{l\cdot\theta}_t\cdot v_tdl.
 \end{align*}
 We have an analogous equation for $\kappa^{\theta}_t(z)-\hat\kappa_t(z)$. 
By the definition of $Y^\theta$,
 \begin{align}
  dY^\theta_t&=\frac{1}{\theta}(dX^{\hat u+\theta v}_t-dX^{\hat u}_t)-dZ_t\nonumber\\
  &=\frac{1}{\theta}\Big((b^{\theta}_t-\hat b_t)dt+\int\limits_\R\kappa^{\theta}_t(z)-\hat\kappa_t(z)\mu(dt,dz)\Big)-dZ_t\nonumber\\
  &=\Big\{\int\limits_0^1\partial_xb^{l\cdot\theta}_t\cdot Y^\theta_{t-}dl+\int\limits_0^1(\partial_xb^{l\cdot\theta}_t-\partial_x\hat b_t) Z_{t-}dl+\int\limits_0^1\partial_yb^{l\cdot\theta}_tE[Y^\theta_t]dl\nonumber\\
  &\quad+\int\limits_0^1(\partial_yb^{l\cdot\theta}_t-\partial_y\hat b_t)E[Z_t]dl+\int\limits_0^1(\partial_ub^{l\cdot\theta}_t-\partial_u\hat b_t)\cdot v_tdl\Big\}dt\label{DynamicsYtheta}\\
  &\quad+\int\limits_\R\Big\{\int\limits_0^1\partial_x\kappa^{l\cdot\theta}_t(z)\cdot Y^\theta_{t-}dl+\int\limits_0^1(\partial_x\kappa^{l\cdot\theta}_t(z)-\partial_x\hat\kappa_t(z)) Z_{t-}dl+\int\limits_0^1\partial_y\kappa^{l\cdot\theta}_t(z)E[Y^\theta_t]dl\nonumber\\
  &\quad+\int\limits_0^1(\partial_y\kappa^{l\cdot\theta}_t(z)-\partial_y\hat \kappa_t(z))E[Z_t]dl+\int\limits_0^1(\partial_u\kappa^{l\cdot\theta}_t(z)-\partial_u\hat\kappa_t(z))\cdot v_tdl\Big\}\mu(dt,dz).\nonumber
 \end{align}
Hereafter we study the convergence to $0$ in $L^2$ of all the terms in the dynamics of $Y^\theta$ that contain a difference. 
We take the term $\int_0^t\int_0^1(\partial_xb^{l\cdot\theta}_t-\partial_x\hat b_t) Z_{t-}dlds$ as an example, the other ones work the same way. First note that, by the definition of $Y^\theta$, the Lipschitzianity of $b$ and $\kappa$ which follows from $(E2')$,
 \begin{align*}
  &E[\sup_{t\in[0,T]}|\theta(Y^\theta_{t}+Z_{t})|^2]=E[\sup_{t\in[0,T]}|X^\theta_{t}-X^{\hat u}_t|^2]\\
  &=E\Big[\sup_{t\in[0,T]}\Big|\int_0^tb^\theta_s-\hat b_sds+\int_0^t\int_\R\kappa^\theta_s(z)-\hat\kappa_s(z)\mu(ds,dz)\Big|^2\Big]\\
  &\leq2(T\vee\ C_1)E\Big[\int_0^T|b^\theta_s-\hat b_s|^2+\lVert\kappa^\theta_s(z)-\hat\kappa_s(z)\rVert_{\lambda_s}^2ds\Big]\\
  &\leq2(T\vee\ C_1)E\Big[\int_0^T6K^2|\theta(Y^\theta_{s-}+Z_{s-})|^2+3K^2E[|\theta(Y^\theta_{s}+Z_{s})|^2]+6K^2\theta^2|v_s|^2ds\Big]
  \end{align*}
  \begin{align*}
  \hspace{2cm}&\leq K_1\int_0^TE[\sup_{t\in[0,s]}|\theta(Y^\theta_t+Z_t)|^2]ds+\theta^2K_2\lVert v\rVert_{\Hh^{\FF}}^2.
 \end{align*}
By an argument in the proof of Theorem \ref{ExistenceAndUniquenessmfSDE}, the function $s\longmapsto E[\sup_{t\in[0,s]}|\theta(Y^\theta_t+Z_t)|^2]=E[\sup_{t\in[0,s]}|X^\theta_t-X^{\hat u}_t|^2]$ is integrable and we can apply Gronwall's inequality to get
 \begin{align*}
  E[\sup_{t\in[0,T]}|\theta(Y^\theta_{t}+Z_{t})|^2]\leq\theta^2K_2\lVert v\rVert_{\Hh^{\FF}}^2e^{K_1T}\rightarrow0\text{ as }\theta\rightarrow0.
 \end{align*}
 Moreover, by Lipschitzianity of $\partial_xb$ (see $(E2')$),
 \begin{align*}
  &E\Big[\sup_{t\in[0,T]}\Big|\int_0^1(\partial_xb^{l\cdot\theta}_t-\partial_x\hat b_t)dl\Big|^2\Big]\\
  &\quad\leq E\Big[\sup_{t\in[0,T]}3L^2\int_0^1|l\theta(Y^\theta_{t-}+Z_{t-})|^2+E[|l\theta(Y^\theta_t+Z_t)|^2]+l^2\theta^2|v_t|^2dl\Big]\\
  &\quad\leq L^2E[\sup_{t\in[0,T]}|\theta(Y^\theta_{t}+Z_{t})|^2]+\theta^2L^2\lVert v\rVert_{\Hh^{\FF}}^2\rightarrow0\text{ as }\theta\rightarrow0.
 \end{align*}
 This shows that $\sup_{t\in[0,T]}|\int_0^1(\partial_xb^{l\cdot\theta}_t-\partial_x\hat b_t)dl|^2$ vanishes  in $L^1$ and therefore also in probability. 
From the fist part of this proof we have that $Z\in \Ss^\FF_2$, hence the continuous mapping theorem yields
 \begin{align*}
  \sup_{t\in[0,T]}\Big(\Big|\int_0^1(\partial_xb^{l\cdot\theta}_t-\partial_x\hat b_t)dl\Big|^2\Big)\cdot\sup_{t\in[0,T]}|Z_t|^2\stackrel{P}{\rightarrow}0,\text{ as }\theta\rightarrow0,
 \end{align*}
 Moreover, note that the family $\sup_{t\in[0,T]}\big(\big|\int_0^1(\partial_xb^{l\cdot\theta}_t-\partial_x\hat b_t)dl\big|^2\big)\cdot \sup_{t\in[0,T]}|Z_t|^2$, ${\theta\in(0,\delta)}$  (for $\delta$ small) is uniformly integrable. This follows from the boundedness of $\partial_xb$ and the fact that $Z\in \Ss^\FF_2$. 
Then we can apply Vitali's theorem and get
 \begin{align*}
  \sup_{t\in[0,T]}\Big(\Big|\int_0^1(\partial_xb^{l\cdot\theta}_t-\partial_x\hat b_t)dl\Big|^2\Big)\cdot\sup_{t\in[0,T]}|Z_t|^2\rightarrow0,\text{ in }L^1(\Omega)\text{ as }\theta\rightarrow0
 \end{align*}
 So all together, for $\theta\rightarrow0$,
 \begin{align*}
  E\Big[\int\limits_0^T\Big(\int\limits_0^1(\partial_xb^{l\cdot\theta}_t-\partial_x\hat b_t) Z_{t-}dl\Big)^2dt\Big]&\leq TE\Big[\sup_{t\in[0,T]}\big(\big|\int_0^1(\partial_xb^{l\cdot\theta}_t-\partial_x\hat b_t)dl\big|^2\big)\cdot\sup_{t\in[0,T]}|Z_t|^2\Big]
  \rightarrow0.
 \end{align*}
 The same arguments apply to all the other terms in the dynamics of $Y^\theta$ that contain a difference. Therefore, by the boundedness of the derivatives assumed in $(E2')$, we have
 \begin{align*}
  E[\sup_{t\in[0,T]}|Y^\theta_t|^2]&\leq G(\theta)+3E\Big[\sup_{t\in[0,T]}\Big|\int_0^t\int\limits_0^1\partial_xb^{l\cdot\theta}_s\cdot Y^\theta_{s-}dl+\int\limits_0^1\partial_yb^{l\cdot\theta}_sE[Y^\theta_s]dlds\Big|^2\Big]\\
  &\quad+3E\Big[\sup_{t\in[0,T]}\Big|\int_0^t \int\limits_\R\int\limits_0^1\partial_x\kappa^{l\cdot\theta}_s(z)\cdot Y^\theta_{s-}dl+\int\limits_0^1\partial_y\kappa^{l\cdot\theta}_s(z)E[Y^\theta_s]dl\mu(ds,dz)  \Big|^2\Big]\\
  &\leq G(\theta)+3K^2(T\vee C_1)\int_0^TE[\sup_{t\in[0,s]}|Y^\theta_t|^2]ds,
 \end{align*}
 where $G(\theta)$ contains all the terms in \eqref{DynamicsYtheta} that contain a difference and therefore vanishes as $\theta$ goes to zero.
 Furthermore, by Theorem \ref{ExistenceAndUniquenessmfSDE}, we have that, for each $\theta>0$,  $\lVert X^\theta\rVert_{S_2}<\infty$, $\lVert X^{\hat u}\rVert_{\Ss_2}<\infty$, and $\lVert Z\rVert_{\Ss_2}<\infty$.
 This implies that also $\lVert Y^\theta\rVert_{\Ss_2}<\infty$. Since $E[\sup_{t\in[0,s]}|Y^\theta_t|^2]\leq\lVert Y^\theta\rVert_{\Ss_2}$, we see that the function $s\longmapsto E[\sup_{t\in[0,s]}|Y^\theta_t|^2]$ is integrable on $[0,T]$. From Gronwall's inequality we obtainthat $E[\sup_{t\in[0,T]}|Y^\theta_t|^2]\rightarrow0$, as $\theta\rightarrow0$, see \eqref{convergence}.
\end{proof}

\begin{remark}
 The last lemma shows that the process $Z$ actually corresponds to
 \begin{align*}
  Z_t:=\frac{d}{d\theta}X^{\hat u+\theta v}_t\Big|_{\theta=0},
 \end{align*}
\end{remark}

\begin{lemma}\label{DerivativeOfJ}
  Let Assumptions \ref{MainAssumptions} be satisfied and let $\hat u\in\A$ be an optimal control. Moreover, let $v\in\Hh^{\FF}$ such that for $\theta$ small enough, $\hat u+\theta v\in\A$. Then 
  \begin{align}\label{DerivativeOfJeq}
   &\frac{d}{d\theta}J(\hat u+\theta v)|_{\theta=0}\\
   &\quad=E\Big[\int\limits_0^T\partial_x\hat f_s\cdot Z_s+\partial_y\hat f_s\cdot E[\partial_x\hat\varphi_s\cdot Z_s]+\partial_u\hat f_s\cdot v_sds+\partial_x\hat g_T\cdot Z_T+\partial_y\hat g_T\cdot E[\partial_x\hat\chi_T\cdot Z_T]\Big].\nonumber
  \end{align}
\end{lemma}

\begin{proof}
 Define a process $Z$ as in Lemma \ref{DerivativeProcesswrtuPsiPhiAffine}. We have that
 \begin{align*}
   &\frac{d}{d\theta}J(\hat u+\theta v)|_{\theta=0}\\
   &=\lim_{\theta\rightarrow0}E\Big[\int\limits_0^T\frac{f(s,\lambda_s,X^{\hat u+\theta v}_{s-},E[\varphi(X^{\hat u+\theta v}_s)],\hat u_s+\theta v_s)-f(s,\lambda_s,X^{\hat u}_{s-},E[\varphi(X^{\hat u}_s)],\hat u_s)}{\theta}ds\\
   &\quad\quad\quad+\frac{g(X^{\hat u+\theta v}_{T},E[\chi(X^{\hat u+\theta v}_{T})])-g(X^{\hat u}_{T},E[\chi(X^{\hat u}_{T})])}{\theta}\Big].
  \end{align*}
 An application of the mean value theorem as in the proof of Lemma \ref{ExistenceAndUniquenessSDEForOurOptimizationProblem} yields the existence of stochastic processes $(\alpha_s)_{s\in[0,1]}$, $(\beta_s)_{s\in[0,1]}$ and random variables $\gamma$ and $\delta$, with values in $[0,1]$, s.t.
  \begin{align*}
   &\frac{d}{d\theta}J(\hat u+\theta v)|_{\theta=0}\\
   &=\lim_{\theta\rightarrow0}E\Big[\int\limits_0^T\frac{1}{\theta}\Big\langle\nabla_{x,y,u}f^{\alpha}_s,\begin{pmatrix}
X^{\hat u+\theta v}_{s-}-X^{\hat u}_{s-}\\                                                                                                                        
E[\partial_x\varphi^\beta_s(X^{\hat u+\theta v}_s-X^{\hat u}_s)]\\
 \theta v_s
 \end{pmatrix}
\Big\rangle ds+\frac{1}{\theta}\Big\langle\nabla g^\gamma,\begin{pmatrix}
 X^{\hat u+\theta v}_T-X^{\hat u}_T\\                                                                                                                        
E[\partial_x\chi^\delta(X^{\hat u+\theta v}_T-X^{\hat u}_T)]
\end{pmatrix}
\Big\rangle\Big]\\
&=\lim_{\theta\rightarrow0}E\Big[\int\limits_0^T\partial_xf^{\alpha}_s\cdot\frac{X^{\hat u+\theta v}_{s-}-X^{\hat u}_{s-}}{\theta}+\partial_yf^{\alpha}_s\cdot E\Big[\partial_x\varphi^\beta_s\cdot\frac{X^{\hat u+\theta v}_s-X^{\hat u}_s}{\theta}\Big]+\partial_uf^{\alpha}_s\cdot v_sds\\
&\quad+\partial_xg^\gamma\cdot\frac{X^{\hat u+\theta v}_{T}-X^{\hat u}_{T}}{\theta}+\partial_yg^\gamma\cdot E\Big[\partial_x\chi^\delta\cdot\frac{X^{\hat u+\theta v}_{T}-X^{\hat u}_{T}}{\theta}\Big]\Big],
  \end{align*}
where
\vspace{-2mm}
\begin{align*}
 f^{\alpha}_s&:=f(s,\lambda_s,X^{\hat u+\alpha_s\theta v}_{s-},E[\varphi(X^{\hat u+\alpha_s\theta v}_s)],\hat u_s+\alpha_s\theta v_s)\\
 \varphi^\beta_s&:=\varphi(X^{\hat u+\beta_s\theta v}_s)\\
 g^\gamma&:=g(X^{\hat u+\gamma\theta v}_{T},E[\chi(X^{\hat u+\gamma\theta v}_{T})])\\
 \chi^\delta&:=\chi(X^{\hat u+\delta\theta v}_{T}).
\end{align*}
Now we prove convergence of the five terms on the right-hand side to the corresponding ones in \eqref{DerivativeOfJeq}. 
As illustration we develop the computations for the first term. Consider the sum
\begin{align}\label{Telescopesum}
 &E\Big[\int\limits_0^T\partial_xf^{\alpha}_s\cdot\frac{X^{\hat u+\theta v}_{s-}-X^{\hat u}_{s-}}{\theta}ds\Big]\\
 &=E\Big[\int\limits_0^T\Big(\partial_xf^{\alpha}_s\cdot\frac{X^{\hat u+\theta v}_{s-}-X^{\hat u}_{s-}}{\theta}-\partial_xf^{\alpha}_s\cdot Z_{s-}\Big)+\Big(\partial_xf^{\alpha}_s\cdot Z_{s-}-\partial_x\hat f_s\cdot Z_{s-}\Big)+\partial_x\hat f_s\cdot Z_sds\Big]\nonumber
\end{align}
The first summand vanishes. In fact, by Holder's inequality
\begin{align*}
 &E\Big[\int\limits_0^T\Big|\partial_xf^{\alpha}_s\cdot\frac{X^{\hat u+\theta v}_{s-}-X^{\hat u}_{s-}}{\theta}-\partial_xf^{\alpha}_s\cdot Z_{s-}\Big|ds\Big]\\
 &\quad\leq {\Big({E\Big[\int\limits_0^T\Big|\partial_xf^{\alpha}_s\Big|^2ds\Big]} \Big)^{1/2}} \cdot {\Big({E\Big[\sup_{s\in[0,T]}\Big|\frac{X^{\hat u+\theta v}_{s-}-X^{\hat u}_{s-}}{\theta}-Z_{s-}\Big|^2\Big]}\Big)^{1/2}} = (I) \cdot (II).
\end{align*}
Thanks to $(O7)$, the first factor is either bounded or, if $\partial_xf^{\alpha}_s$ is Lipschitz in $x$, $y$ and $z$, then
\begin{align*}
 &E\Big[\int\limits_0^T\Big|\partial_xf^{\alpha}_s\Big|^2ds\Big]\leq E\Big[\int\limits_0^T2\Big|\partial_x\hat f_s\Big|^2+2\Big|\partial_xf^{\alpha}_s-\partial_x\hat f_s\Big|^2ds\Big]\\
 &\quad\leq E\Big[\int\limits_0^T6M^2\Big(1+|X^{\hat u}_{s-}|^2+|E[\varphi(X^{\hat u}_s)]|^2+|\hat u_{s}|^2\Big)\\
 &\quad\quad+6L_{\partial_x f}\Big(|X^{\hat u}_{s-}-X^{\hat u+\alpha_s\theta v}_{s-}|^2+|E[\varphi(X^{\hat u}_s)]-E[\varphi(X^{\hat u+\alpha_s\theta v}_s)]|^2+|\alpha_s\theta v_s|^2\Big)ds\Big]\\
 &\quad\leq C_1+C_2\lVert X^{\hat u}\rVert_{\Ss^\FF_2}^4+C_3\lVert X^{\hat u}-X^{\hat u+\alpha_s\theta v}\rVert_{\Ss^\FF_2}^4+\lVert\hat u\rVert_{\Hh^\FF}^2+\theta^2\lVert v\rVert_{\Hh^\FF}^2.
\end{align*}
Here we used that
\begin{align*}
 |\varphi(x_1)-\varphi(x_2)|&\leq(|\partial_x\varphi(x_1)|\vee|\partial_x\varphi(x_2)|)|x_1-x_2|\leq(|\partial_x\varphi(x_1)|+L_{\partial_x\varphi}|x_2-x_1|)|x_1-x_2|\\
 &\leq(|\partial_x\varphi(0)|+L_{\partial_x\varphi}|x_1|+L_{\partial_x\varphi}|x_2-x_1|)|x_1-x_2|,
\end{align*}
, as $\varphi$ is concave and $\partial_x\varphi$ is Lipschitz, this implies that
\begin{align*}
 |E[\varphi(X^{\hat u}_s)]|^2&\leq (|\varphi(0)|+E[|\varphi(0)-\varphi(X^{\hat u}_s)|])^2\\
 &\leq(|\varphi(0)|+|\partial_x\varphi(0)|E[||X^{\hat u}_s|]+L_{\partial_x\varphi}E[||X^{\hat u}_s|^2])^2\\
 &\leq 3|\varphi(0)|^2+3|\partial_x\varphi(0)|^2\lVert X^{\hat u}\rVert_{\Ss^\FF_2}^2+3L_{\partial_x\varphi}^2\lVert X^{\hat u}\rVert_{\Ss^\FF_2}^4
\end{align*}
and accordingly
\begin{align}
 |E[\varphi(X^{\hat u}_s)] &-E[\varphi(X^{\hat u+\alpha_s\theta v}_s)]|^2
 \leq3|\partial_x\varphi(0)|^2\lVert X^{\hat u}-X^{\hat u+\alpha_s\theta v}\rVert_{\Ss^\FF_2}^2 \notag \\
 &\quad+3L_{\partial_x\varphi}^2\lVert X^{\hat u}\rVert_{\Ss^\FF_2}^2\lVert X^{\hat u}-X^{\hat u+\alpha_s\theta v}\rVert_{\Ss^\FF_2}^2
 +3L_{\partial_x\varphi}^2\lVert X^{\hat u}-X^{\hat u+\alpha_s\theta v}\rVert_{\Ss^\FF_2}^4
 \label{servizio}
\end{align}
Being $\lVert X^{\hat u}-X^{\hat u+\alpha_s\theta v}\rVert_{\Ss^\FF_2}\rightarrow0$, then $(I)$ is bounded. Lemma \ref{DerivativeProcesswrtuPsiPhiAffine} yields the convergence of $(II)$ towards zero as $\theta\rightarrow0$. Therefore, the first summand in \eqref{Telescopesum} vanishes.

\vspace{2mm}
Now we consider the second summand in \eqref{Telescopesum}. By Holder's inequality we have
\begin{align*}
 E\Big[\int\limits_0^T(\partial_xf^{\alpha}_s-\partial_x\hat f_s)\cdot Z_{s-}ds\Big]\leq\lVert Z\rVert_{\Ss^\FF_2}\Big({E\Big[\int\limits_0^T|\partial_xf^{\alpha}_s-\partial_x\hat f_s|^2ds\Big]}\Big)^{1/2}.
\end{align*}
Note that in case $\partial_x f$ is Lipschitz, we get convergence by the same arguments as before. In case $\partial_x\hat f$ is only bounded and does not depend on $(s,\lambda)$, we have that $X^{\hat u+\alpha_s\theta v}\rightarrow X^{\hat u}$ in $\Ss^\FF_2$ and thus also w.r.t. the finite measure $\text{Leb}\otimes P$ and by the same argumentation $\hat u+\alpha_s\theta v$ converges w.r.t. $\text{Leb}\otimes P$ towards $\hat u$. Also, by similar arguments as before $E[\varphi(X^{\hat u+\alpha_s\theta v}_s)]\rightarrow E[\varphi(X^{\hat u}_s)]$ as $\theta\rightarrow0$. By Assumption $(O1)$, $\partial_x f$ is continuous and the continuous mapping theorem yields
\begin{align*}
 \partial_xf^{\alpha}\rightarrow\partial_x\hat f,\text{ as }\theta\rightarrow0\text{ in measure w.r.t. }Leb\otimes P.
\end{align*}
Moreover, the boundedness of $\partial_xf$ implies that the family $(|\partial_xf^{\alpha}-\partial_x\hat f_s|^2)_{\theta>0}$ is uniformly integrable w.r.t. $\text{Leb}\otimes P$. Therefore, Vitali's theorem yields that
\begin{align*}
 E\Big[\int\limits_0^T|\partial_xf^{\alpha}_s-\partial_x\hat f_s|^2ds\Big]\rightarrow0,\text{ as }\theta\rightarrow0.
\end{align*}
This proves that also the second summand in \eqref{Telescopesum} converges to $0$ and so we get
\begin{align*}
 &E\Big[\int\limits_0^T\partial_xf^{\alpha}_s\cdot\frac{X^{\hat u+\theta v}_{s-}-X^{\hat u}_{s-}}{\theta}ds\Big]\rightarrow E\Big[\int\limits_0^T\partial_x\hat f_s\cdot Z_sds\Big].
\end{align*}
The same applies to all the other terms in the representation of $\frac{d}{d\theta}J(\hat u+\theta v)|_{\theta=0}$.
\end{proof}

\begin{lemma}\label{DisturbedStrategy}
 Let $\hat u\in\A$ be an optimal control and let $v\in U$ be fixed. Moreover, let $t_1<t_2\in[0,T]$ and $S\in\F_{t_1}$. The strategy
 \begin{align}
  u_t(\omega)=\hat u_t(\omega)\ind_{[0,t_1]\cup(t_2,T]}(t)+v\ind_{S\times(t_1,t_2]}(\omega,t)+\hat u_t(\omega)\ind_{S^c\times(t_1,t_2]}(\omega,t)
 \end{align}
 is an admissible strategy.
\end{lemma}

\begin{proof}
 We verify Definition \ref{Admissibility}. Clearly, $u\in\Hh^\FF$ and also, by construction, $u$ takes values in $U$. In order to check the integrability of $f$, $g$, $\partial_i f$ and $\partial_i g$, $i=x,y,u$, we use Gronwall's inequality obtaining
 \begin{align*}
  \lVert X^{\hat u}- X^u\rVert_{\Ss^\FF_2}^2\leq \tilde K(v^2+\lVert\hat  u\rVert_{\Hh^\FF}^2)(t_2-t_1)e^{\tilde KT}
 \end{align*}
 for some constant $\tilde K$. This, together with $(O6)$, Remark \ref{AlternativeA9}, and the relation of type \eqref{servizio} for 
$
  |E[\varphi(X^{\hat u}_s)]-E[\varphi(X^{u}_s)]|^2
$
prove the required integrability conditions.
\end{proof}

\begin{theorem}\label{NecessaryMaxPrinciple}
\begin{itemize}
 \item[(a)] Let $\hat u\in\A$ be an optimal control. Then, there exists a solution $(\hat p,\hat q)$ of the corresponding adjoint mean-field BSDE \eqref{adjointeq} such that for all $v\in U$:
 \begin{align}\label{NecessaryMaxPrincipleEq1}
  \partial_uH^\FF(t,\lambda_t,X^{\hat u}_{t-},E[\varphi(X^{\hat u}_t)],E[X^{\hat u}_t],\hat u_t,\hat p_{t-},\hat q_t)\cdot(v-u_t)\leq0.
 \end{align}
 \item[(b)] Let, on the other hand $\hat u\in\A$ such that
 \begin{align}\label{NecessaryMaxPrincipleEq2}
  \partial_uH^\FF(t,\lambda_t,X^{\hat u}_{t-},E[\varphi(X^{\hat u}_t)],E[X^{\hat u}_t],\hat u_t,\hat p_{t-},\hat q_t)=0.
 \end{align}
 Then, $\hat u$ is a critical point for $J$, i.e. $\frac{d}{d\theta}J(\hat u+\theta(u-\hat u))|_{\theta=0}=0$ for all $u\in\A$.
\end{itemize}
\end{theorem}

\begin{proof}
Part $(a)$. Since $\hat u$ is optimal and $J$ is $C^1$, $\hat u$ must be a critical point for $J$, i.e. by Lemma \ref{DerivativeOfJ},
 \begin{align*}
  0&\geq\frac{d}{d\theta}J(\hat u+\theta v)|_{\theta=0}\\
  &= E\Big[\int\limits_0^T\partial_x\hat f_s\cdot Z_s+\partial_y\hat f_s\cdot E[\partial_x\hat\varphi_s\cdot Z_s]+\partial_u\hat f_s\cdot v_sds+\partial_x\hat g_T\cdot Z_T+\partial_y\hat g_T\cdot E[\partial_x\hat\chi_T\cdot Z_T]\Big]\\
  &=E\Big[\int\limits_0^T\partial_x\hat f_s\cdot Z_s+\partial_y\hat f_s\cdot E[\partial_x\hat\varphi_s\cdot Z_s]+\partial_u\hat f_s\cdot v_sds+\hat p_T\cdot Z_T-\hat p_0\cdot Z_0\Big].
 \end{align*}
In the last equality, we used the terminal condition in \eqref{adjointeq}, \eqref{StandardTrickExpectations}, and $Z_0=0$. The product rule together with \eqref{adjointeq} and Lemma \ref{DerivativeProcesswrtuPsiPhiAffine} yields
 \begin{align*}
  0&\geq 
  E\Big[\int\limits_0^T\partial_x\hat f_s\cdot Z_s+\partial_y\hat f_sE[\partial_x\hat\varphi_s\cdot Z_s]+\partial_u\hat f_s\cdot v_sds+\int\limits_0^T\hat p_{s-}\{\partial_x\hat b_s\cdot Z_{s-}+\partial_y\hat b_sE[Z_s]+\partial_u\hat b_s\cdot v_s\}ds\\
  &\quad-\int\limits_0^TZ_{s-}\Big\{\partial_x\hat f_s+\partial_x\hat b_s\cdot\hat p_{s-}+\partial_x\hat\kappa_s(0)\hat q_s(0)\lambda^B_s+\int\limits_{\R_0}\partial_x\hat\kappa_s(z)\hat q_s(z)\lambda^H_s\nu(dz)+E[\partial_y\hat f_s]\partial_x\hat\varphi_s\\
  &\quad+E[\partial_y\hat b_s\cdot\hat p_{s-}]+E\Big[\partial_y\hat\kappa_s(0)\hat q_s(0)\lambda^B_s+\int\limits_{\R_0}\partial_y\hat\kappa_s(z)\hat q_s(z)\lambda^H_s\nu(dz)\Big]\Big\}ds\\
  &\quad+\int\limits_0^T\Big\{\hat q_s(0)\cdot\Big(\partial_x\hat\kappa_s(0)\cdot Z_{s-}+\partial_y\hat\kappa_s(0)\cdot E[Z_s]+\partial_u\hat\kappa_s(0)\cdot v_s\Big)\lambda^B_s\\
  &\quad+\int\limits_\R\hat q_s(z)\cdot\Big(\partial_x\hat\kappa_s(z)\cdot Z_{s-}+\partial_y\hat\kappa_s(z)\cdot E[Z_s]+\partial_u\hat\kappa_s(z)\cdot v_s\Big)\lambda^H_s\nu(dz)\Big\}ds\Big]\\
  &=E\Big[\int_0^T\Big\{\partial_u\hat f_s\cdot v_s+\hat p_{s-}\partial_u\hat b_s\cdot v_{s}+\hat q_s(0)\partial_u\hat\kappa_s(0)\cdot v_s\lambda^B_s+\int\limits_\R\hat q_s(z)\partial_u\hat\kappa_s(z)\cdot v_s\lambda^H_s\nu(dz)\Big\}ds\Big].
 \end{align*}
 As $v$ and $Z$ are $\FF$-adapted, the only terms that are $\GG$-adapted, are $\hat p$ and $\hat q$. Applying Fubini's theorem, the tower property and again Fubini's theorem we achieve
 \begin{align*}
  0&\geq E\Big[\int_0^T\Big\{\partial_u\hat f_s\cdot v_s+E[\hat p_{s-}|\F_s]\partial_u\hat b_s\cdot v_{s}+E[\hat q_s(0)|\F_s]\partial_u\hat\kappa_s(0)\cdot v_s\lambda^B_s\\
  &\quad+\int\limits_\R E[\hat q_s(z)|\F_s]\partial_u\hat\kappa_s(z)\cdot v_s\lambda^H_s\nu(dz)\Big\}ds\Big]\\
  &=E\Big[\int_0^T\partial_uH^\FF(s,\lambda_s,X^{\hat u}_{s-},E[\varphi(X^{\hat u}_s)],E[X^{\hat u}_s],\hat u_s,\hat p_{s-},\hat q_s)\cdot v_sds\Big].
 \end{align*}
 By Lemma \ref{DisturbedStrategy}, the strategy $u$ given by $$u_t(\omega)=\hat u_t(\omega)\ind_{[0,t_1]\cup(t_2,T]}(t)+v\ind_{S\times(t_1,t_2]}(\omega,t)+\hat u_t(\omega)\ind_{S^c\times(t_1,t_2]}(\omega,t)$$
 is admissible for every $v\in U$, $S\in\F_{t_1}$ and all $t_1<t_2\in[0,T]$. Defining $v_t:=u_t-\hat u_t$, the convexity of $\A$ implies that also $\hat u_t+\theta v_t\in\A$ for all $\theta\in[0,1]$. Thus, $v_t$ satisfies the conditions above and we get for all $t_1<t_2\in[0,T]$
 \begin{align*}
  0&\geq E\Big[\int_{t_1}^{t_2}\ind_S\partial_uH^\FF(s,\lambda_s,X^{\hat u}_{s-},E[\varphi(X^{\hat u}_s)],E[X^{\hat u}_s],\hat u_s,\hat p_{s-},\hat q_s)\cdot(v-\hat u_s)ds\Big].
 \end{align*}
 Letting $t_2\downarrow t_1$, this implies that $$E[\ind_S\partial_uH^\FF(s,\lambda_s,X^{\hat u}_{s-},E[\varphi(X^{\hat u}_s)],E[X^{\hat u}_s],\hat u_s,\hat p_{s-},\hat q_s)\cdot(v-\hat u_s)]\leq 0$$ for all $S\in\F_s$ a.e. $s\in[0,T]$ and $(a)$ is proved.
 
 \vspace{2mm}
Part $(b)$. Let now $\partial_uH^\FF(s,\lambda_s,X^{\hat u}_{s-},E[\varphi(X^{\hat u}_s)],E[X^{\hat u}_s],\hat u_s,\hat p_{s-},\hat q_s)=0$. Then, putting $v_t:=u_t-\hat u_t$, it directly follows from the computations in $(a)$ that, for all $u\in\A$, $$\frac{d}{d\theta}J(\hat u+\theta(u-\hat u))|_{\theta=0}=0.$$ 
\end{proof}

\subsection{Example: A centralised control in an economy of specialised sectors}
Consider an economy of $N$ specialised sectors, all of them with more or less the same debt and
having comparably sized volumes (as a motivation for this assumption, see Ricardo's theory of comparative advantages). 
We identify one sector with one (leading) agent in the economy.
Each agent $i$ can sell bonds at a rate $(r^i_t)_{t\in[0,T]}$, whose dynamics follows a generalised Vasicek model:
\begin{align}
 d r^i_t&=\theta_t\{(\bar r_t-u_t)-r^i_t\}dt+\int\limits_\R\sigma_t(z)\mu^i(dt,dz),\, j=1,\cdots,\,N\label{Vasicek}\\
 \bar r_t&:=\frac{1}{\sqrt{N}}\sum_{j=1}^N r^j_t.\nonumber
\end{align}
The martingale random fields $\mu^i(dt,dz)$ are iid, as the sectors are specialised, each one on a different industry, but the volatilities are the same, as we assume equal economic strength. The process $(\theta_t)_{t\in[0,T]}$ is assumed to be positive, predictable and bounded by some constant $K$. We also assume that $\lVert\sigma_s(\cdot)\rVert_{\lambda_s}<K$. The term $u_t\geq0$ is a control term that models the influence of the central bank regulating this economy. The central bank can buy a basket of bonds (in this example equally weighted) in order to lower the average interest rate. So the term $\bar r_t-u_t$ models the target average rate at time $t$. By linearity and the independent noises in \eqref{Vasicek} and a propagation of chaos argument, we have that, as $N\rightarrow\infty$, each state's dynamics behaves like:
\begin{align}\label{MeanFieldVasicek}
 d r_t&=\theta\{(E[r_t]-u_t)-r_t\}dt+\int\limits_\R\sigma_t(z)\mu(dt,dz),
\end{align}
for a martingale random measure $\mu\stackrel{d}{=}\mu^1$. The main focus of the central bank is  on keeping the currency stable and it is necessary to keep $u_t$ small over time, while still promoting liquidity in the economy.Then the central bank faces the following optimisation problem
\begin{align}
 J(\hat u)&=\max_{u}J(u) \notag\\
  J(u)&=E\Big[\int\limits_0^T-(u_t)^2-(E[r_t])^2-(r_t)^2dt\Big]. \label{ObjectiveInterest RateExample}
\end{align}
In the notation of the previous sections, the functions $b$, $\kappa$, $f$ and $g$ are given by
\begin{align*}
 b(t,\lambda_t,x,y,u)&=\theta_t(-x+y-u)\\
 \kappa(t,z,\lambda_t,x,y,u)&=\sigma_t(z)\\
 f(t,\lambda_t,x,y,u)&= -x^2-y^2-u^2\\
 g(x,y)&=0
\end{align*}
We can solve this optimisation problem explicitly, as it is quadratic (see \cite{CDL} for the classical Brownian case) applying the results before.
First we need to check Assumptions \ref{MainAssumptionsForwardDynamics}, \ref{MainAssumptions}, and Definition \ref{Admissibility}.

Concerning assumption $(E1')$, defining $b_0(s,\lambda_s,\omega)=\theta_s(\omega)$ and $b_1(s,\lambda,x,y,u)=-x+y-u$ and $\kappa_0(s,z,\lambda_s,\omega)=\sigma_s(z,\omega)$, $\kappa_1(s,z,\lambda,x,y,u)=1$, we have $b(s,\lambda_s,x,y,u)=b_0(s,\lambda_s)b_1(s,\lambda_s,x,y,u)$ and $\kappa(s,\lambda_s,x,y,u)=\kappa_0(s,z,\lambda_s)\kappa_1(s,z,\lambda_s,x,y,u)$. Since $|\partial_ib|=|\theta_s|\leq K$, and  $\lVert\sigma_s(\cdot)\rVert_{\lambda_s}<K$ does not depend on $(x,y,u)$, $(E2')$ holds. $(E3')$ is equally simple. About the conditions on $f$, $g$, $\varphi=\id$ and $\chi=0$, note that $f$ is obviously $C^1$ in $(x,y,u)$, $g$ is obviously concave, $\varphi$ and $\chi$ are affine and thus $\partial_x\varphi$, $\partial_x\chi$ are Lipschitz. This proves that conditions $(O1)-(O5)$ are all satisfied. About condition $(A1)$, we know that $(\omega,s)\mapsto X^u_{s-}(\omega)\in L^2([0,T]\times \Omega)$ since we only allow for solutions of the mean-field SDE in $\Ss^\FF_2$ and, by Theorem \ref{ExistenceAndUniquenessmfSDE}, there is exactly one such solution. 
For $u\in\Hh^\FF$, Jensen's inequality implies that
\begin{align*}
s\mapsto f(s,\lambda_s,X^u_{s-},E[X^u_s],u_s)=-|X^u_{s-}|^2-|E[X^u_s]|^2-|u_s|^2\in L^2(\Omega\times[0,T]),
\end{align*}
which proves $(A1)$. Condition $(A2)$ is proved similarly.

\vspace{2mm}
We first apply the necessary maximum principle Theorem \ref{NecessaryMaxPrinciple}). Note that the additional conditions $(O6)$ and $(O7)$ are obviously satisfied and the $\FF$-Hamiltonian 
\begin{align*}
 H^\FF(t,\lambda_t,x,y_1,y_2,u,\hat p_{t-},\hat q_t)&=-x^2-y_1^2-u^2+\theta_t(-x+y_2-u)E[\hat p_{t-}|\F_t]\\
 &+\sigma(t,0)E[\hat q_t(0)|\F_t]\lambda^B_t +\int\limits_{\R_0}\sigma(t,z)E[\hat q_t(z)|\F_t]\lambda^H_t\nu(dz).
\end{align*}
Then we get the following candidate for the optimal control
$$
 \hat u_t=-\frac{\theta_t}{2}E[\hat p_{t-}|\Ff_t].
$$
Notice that it makes sense to restrict only to nonnegative values, as short selling is not allowed by the central bank. For $\hat u_t:=-\frac{\theta_t}{2}E[\hat p_{t-}|\Ff_t]>0$, we get
\[\begin{split}
H^\FF(t,\lambda_t,x,y_1,y_2,\hat u_t,\hat p_{t-},\hat q_t)&=h(t,\lambda_t,x,y_1,y_2)\\
&=-(x^2+y_1^2)+\theta_t(y_2-x)-\frac{|\theta_t|^2}{4}|E[\hat p_{t-}|\Ff_t]|^2
\\
 &+\sigma(t,0)E[\hat q_t(0)|\F_t]\lambda^B_t +\int\limits_{\R_0}\sigma(t,z)E[\hat q_t(z)|\F_t]\lambda^H_t\nu(dz),
\end{split}\]
which is a concave function in $(x,y_1,y_2)$. 
Then $\hat u$ is optimal by the sufficient maximum principle Theorem \ref{SufficientMaxPrinciple}.
The boundary $u=0$ must be checked separately.

\section{Appendix}
{\bf Proof that the mapping $\Psi$ in \eqref{psi} is a contraction.}
Fix $\beta>0$. We define the norm $\lVert\cdot\rVert_{\beta}$ on $ L^2_{ad}(\GG) \times \Ii$ by
\begin{align*}
 \lVert(Y,Z)\rVert_{\beta}:=\Big(E\Big[\int\limits_0^Te^{\beta s}(|Y_s|^2+\lVert Z_s\rVert^2_{\lambda_s})ds\Big]\Big)^{\frac{1}{2}}   
\end{align*}
which is equivalent to the canonical one. 
Let $(y^{(1)},z^{(1)}),\,(y^{(2)},z^{(2)})\in L^2_{ad}(\GG) \times\Ii$ be two given inputs and define $(Y^{(1)},Z^{(1)}):=\Psi(y^{(1)},z^{(1)})$, $(Y^{(2)},Z^{(2)}):=\Psi(y^{(2)},z^{(2)})$, which are indeed the corresponding solutions of \eqref{RelaxedmfBSDE}. Furthermore, define
\begin{align*}
 \hat Y:=Y^{(1)}-Y^{(2)},\,\hat y:=y^{(1)}-y^{(2)},\,\hat Z:=Z^{(1)}-Z^{(2)},\,\hat z:=z^{(1)}-z^{(2)}.
\end{align*}
Then $(\hat Y,\hat Z)$ satisfies the BSDE
\begin{align*}
 \begin{cases}
  d\hat Y_t&=E'\Big[h\Big(t,\lambda_t,\lambda'_t,Y^{(1)}_t,(y^{(1)}_t)',Z^{(1)}_t,(z^{(1)}_t)'\Big)-h\Big(t,\lambda_t,\lambda'_t,Y^{(2)}_t,(y^{(2)}_t)',Z^{(2)}_t,(z^{(2)}_t)'\Big)\Big]dt\\
  &\quad+\int_\R Z^{(1)}_t(z)-Z^{(2)}_t(z)\mu(dt,dz)\\
  \hat Y_T&=F-F=0.
 \end{cases}
\end{align*}
The application of Ito's formula on $e^{\beta s}|\hat Y_s|^2$ yields
\begin{align*}
 &0\geq E\Big[e^{\beta\cdot T}|\hat Y_T|^2-e^{\beta\cdot0}|\hat Y_0|^2\Big]=E\Big[\int\limits_0^T\beta e^{\beta s}|\hat Y_{s-}|^2ds+\int\limits_0^T\int\limits_\R2e^{\beta s}\hat Y_{s-}\hat Z_s(\xi)\mu(ds,d\xi)
\\
 &\quad+\int\limits_0^T2e^{\beta s}\hat Y_{s-}E'\Big[h\Big(s,\lambda_s,\lambda'_s,Y^{(1)}_s,(y^{(1)}_s)',Z^{(1)}_s,(z^{(1)}_s)'\Big)-h\Big(s,\lambda_s,\lambda'_s,Y^{(2)}_s,(y^{(2)}_s)',Z^{(2)}_s,(z^{(2)}_s)'\Big)\Big]ds\\
 &\quad+\frac{1}{2}\int\limits_0^T2e^{\beta s}|\hat Z_s(0)|^2\lambda^B_sds\quad+\int\limits_0^T\int\limits_{\R_0}
 \{
 e^{\beta s}(|\hat Y_{s-}+\hat Z_s(\xi)|^2-|\hat Y_{s-}|^2)-2e^{\beta s}\hat Y_{s-}\hat Z_s(\xi)
 \}
 \lambda^B_s\nu(d\xi)ds\Big]
\end{align*}
Since $Z^{(1)},\,Z^{(2)}\in\mathcal I$, then the process $M_t:=\int_0^t\int_\R \hat Z_s(z)\mu(ds,dz)$ is a martingale. Since the filtration $\GG$ is right continuous (see \cite[Lemma 2.4]{DiS}), Doob's Regularization Theorem (see, e.g. \cite[Theorem 6.27]{K}) implies that $M$ has a c\`adl\`ag version and, being the integral w.r.t. $ds$ continuous, we conclude that $Y$ has a c\`adl\`ag version. 
Hence the c\`adl\`ag version of $Y$ has only countably many discontinuities, we can replace the $\hat Y_{s-}$ by $\hat Y_{s}$ in the integrals w.r.t. $ds$. Rearranging terms and the Lipschitzianity of $h$, given by $(C3)$ yields
\begin{align*}
 &E\Big[\int\limits_0^T\beta e^{\beta s}|\hat Y_{s}|^2ds+\int\limits_0^Te^{\beta s}\lVert\hat Z_{s}\rVert_{\lambda_s}^2ds\Big]\\
 &\leq-E\Big[\int\limits_0^T2e^{\beta s}\hat Y_{s}E'\Big[h\Big(s,\lambda_s,\lambda'_s,Y^{(1)}_s,(y^{(1)}_s)',Z^{(1)}_s,(z^{(1)}_s)'\Big)-h\Big(s,\lambda_s,\lambda'_s,Y^{(2)}_s,(y^{(2)}_s)',Z^{(2)}_s,(z^{(2)}_s)'\Big)\Big]ds\Big]\\
 &\leq E\Big[\int\limits_0^T2e^{\beta s}|\hat Y_{s}|E'\Big[\Big|h\Big(s,\lambda_s,\lambda'_s,Y^{(1)}_s,(y^{(1)}_s)',Z^{(1)}_s,(z^{(1)}_s)'\Big)-h\Big(s,\lambda_s,\lambda'_s,Y^{(2)}_s,(y^{(2)}_s)',Z^{(2)}_s,(z^{(2)}_s)'\Big)\Big|\Big]ds\Big]
 \end{align*}
 \begin{align*}
 \hspace{2cm}
 &\leq E\Big[\int\limits_0^T2Ke^{\beta s}|\hat Y_{s}|E'\Big[|Y^{(1)}_s-Y^{(2)}_s|+|(y^{(1)}_s)'-(y^{(2)}_s)'|+\lVert Z^{(1)}_s-Z^{(2)}_s\rVert_{\lambda_s}+\lVert (z^{(1)}_s)'-(z^{(2)}_s)'\rVert_{\lambda'_s}\Big]ds\Big]
\end{align*}
By the definition of the operator $E'$, we have 
\begin{align*}
 &E'[|Y^{(1)}_s-Y^{(2)}_s|]=|Y^{(1)}_s-Y^{(2)}_s|=|\hat Y_s|\\
 &E'[|(y^{(1)}_s)'-(y^{(2)}_s)'|]=E[|y^{(1)}_s-y^{(2)}_s|]=E[|\hat y_s|]\\
 &E'[\lVert Z^{(1)}_s-Z^{(2)}_s\rVert_{\lambda_s}]=\lVert Z^{(1)}_s-Z^{(2)}_s\rVert_{\lambda_s}=\lVert \hat Z_s\rVert_{\lambda_s}\\
 &E'[\lVert (z^{(1)}_s)'-(z^{(2)}_s)'\rVert_{\lambda_s}]=E[\lVert z^{(1)}_s-z^{(2)}_s\rVert_{\lambda_s}]=E[\lVert \hat z_s\rVert_{\lambda_s}].
\end{align*}
Making use of the fact that $2ab\leq ka^2+\frac{1}{k}b^2$ for all $a,b\in\R$ and all $k>0$, and choosing $k:=16K$, $a:=|\hat Y_s|$, $b=(|\hat Y_s|+E[|\hat y_s|]+\lVert \hat Z_s\rVert_{\lambda_s}+E[\lVert \hat z_s\rVert_{\lambda_s}])$, we get
\begin{align*}
 &E\Big[\int\limits_0^T\beta e^{\beta s}|\hat Y_{s}|^2ds+\int\limits_0^Te^{\beta s}\lVert\hat Z_{s}\rVert_{\lambda_s}^2ds\Big]\\
 &\leq 16K^2E\Big[\int\limits_0^Te^{\beta s}|\hat Y_{s}|^2ds\Big]+\frac{1}{16}E\Big[\int\limits_0^Te^{\beta s}(|\hat Y_s|+E[|\hat y_s|]+\lVert \hat Z_s\rVert_{\lambda_s}+E[\lVert \hat z_s\rVert_{\lambda_s}])^2ds\Big]\\
 &\leq 16K^2E\Big[\int\limits_0^Te^{\beta s}|\hat Y_{s}|^2ds\Big]+\frac{1}{4}E\Big[\int\limits_0^Te^{\beta s}|\hat Y_s|^2ds\Big]+\frac{1}{4}E\Big[\int\limits_0^Te^{\beta s}|\hat y_s|^2ds\Big]
\\
 \hspace{1cm}
 &\quad+\frac{1}{4}E\Big[\int\limits_0^Te^{\beta s}\lVert \hat Z_s\rVert_{\lambda_s}^2ds\Big]+\frac{1}{4}E\Big[\int\limits_0^Te^{\beta s}\lVert \hat z_s\rVert_{\lambda_s}^2ds\Big],
\end{align*}
where we also used that $(\sum_{i=1}^na_i)^2\leq n\sum_{i=1}^na_i^2$ and $E[X]^2\leq E[X^2]$. This yields
\begin{align*}
 &(\beta-16K^2-\frac{1}{4})E\Big[\int\limits_0^Te^{\beta s}|\hat Y_{s}|^2ds\Big]+\frac{3}{4}E\Big[\int\limits_0^Te^{\beta s}\lVert\hat Z_{s}\rVert_{\lambda_s}^2ds\Big]\leq\frac{1}{4}E\Big[\int\limits_0^Te^{\beta s}(|\hat y_s|^2+\lVert \hat z_s\rVert_{\lambda_s}^2)ds\Big]
\end{align*}
Choosing $\beta=16K^2+1>0$, we finally get
\begin{align*}
 \lVert(\hat Y,\hat Z)\rVert_{\beta}=E\Big[\int\limits_0^Te^{\beta s}(|\hat Y_{s}|^2ds+\lVert\hat Z_{s}\rVert_{\lambda_s}^2)ds\Big]\leq\frac{1}{3}E\Big[\int\limits_0^Te^{\beta s}(|\hat y_s|^2+\lVert \hat z_s\rVert_{\lambda_s}^2)ds\Big]=\frac{1}{3}\lVert(\hat y,\hat z)\rVert_{\beta}.
\end{align*}
By this we see that $\Psi$ is a contraction.

\bigskip
\begin{center}
{\bf Acknowledgements}
\end{center}
Financial support from the Norwegian Research Council of the ISP project 239019 "FINance, INsurance, Energy, Weather and STOCHastics"  (FINEWSTOCH) is greatly acknowledged.



\begin{thebibliography}{10mm}

\bibitem{AD} Andersson, D., Djehiche, B. (2011): A maximum principle for SDEs of mean-field type. 
{\it Appl. Math. Optim.}, 63, 311--356.

\bibitem{Applebaum} Applebaum D. (2009): {\it L\'evy Processes and Stochastic Calculus}. Cambridge University Press.

\bibitem{BFY}
Bensoussan, A., Frehse, J., and Yam, P. (2013): {\it Mean Field Games and Mean Field Type Control Theory}. Springer.

\bibitem{BLP} Buckdahn, R., Li, J., Peng, S. (2009): Mean-field backward stochastic differential equations and related partial differential equations. {\it Stochastic Processes and their Applications}, 119, 3133--3154.

\bibitem{walsh}
Cairoli, R., Walsh, J. (1975): Stochastic integrals in the plane.
{\it Acta Math.}, 134, 111--183.

\bibitem{CD2}
Carmona, R., Delarue, F. (2015): Forward-backward stochastic differential equations and controlled McKean-Vlasov Dynamics. {\it Annals of Probability}, 43, 2647-2700.

\bibitem{CDL}
Carmona R., Delarue F., Lachapelle, A. (2013): Control of McKean?Vlasov dynamics versus mean field games. Mathematics and Financial Economics. {\it Math. Financ. Econ.} 7, 131?166.

\bibitem{DiE} Di Nunno, G., Eide, I.B. (2010): Minimal-variance hedging in large financial markets: random fields approach. {\it Stoch. Anal. Appl.}, 28, 54--85.

\bibitem{DiS} Di Nunno, G., Sjursen, S. (2014): BSDEs driven by time-changed L\'evy noises and optimal control. {\it Stochastic Processes and their Applications}, {124}, 1679--1709.

\bibitem{DiS2} Di Nunno, G., Sjursen, S. (2013): On Chaos Representation and Orthogonal Polynomials for the Doubly Stochastic Poisson Process. In 
{\it Seminar on Stochastic Analysis, Random Fields and Applications VII}. Editors: R.C. Dalang, M. Dozzi, and F. Russo. Springer, Basel, Pages 23-54.

\bibitem{Grigelionis} Grigelionis B. (1975): Characterisation of stochastic processes with conditionally independent increments. {\it Lithuanian Mathematical Journal}, 15, 562--567.

\bibitem{HMC}
Huang, M., Malham\'e, R.P., and Caines, P.E. (2006): Large population stochastic dynamic games: Closed-loop McKean?Vlasov systems and the Nash certainty equivalence principle. {\it Commun. Inf. Syst.}, 6, 221?251.

\bibitem{JMW} Jourdain, B., M\'el\'eard, S., Woyczynski, W. (2008): Nonlinear SDEs driven by L\'evy processes and related PDEs. {\it Alea}, 4, 1-29.

\bibitem{LL}
Lasry, J.-M., Lions, P.-L. (2007): Mean field games. {\it Japan. J. Math.}, 2, 229-260. 

\bibitem{K} Kallenberg, O. (1997): {\it Foundations of Modern Probability}. Springer

\bibitem{Serfozo} Serfozo, R.F. (1972): Processes with conditional stationary independent increments. 
{\it Journal of Applied Probability}, 9, 303--315.

\end{thebibliography}
\end{document}